\documentclass[a4paper,11pt]{article}
\usepackage[T1]{fontenc} % fonts have 256 characters instead of 128
\usepackage[utf8]{inputenc} % input file can use UTF8, e.g. ö is recognised
\usepackage[english]{babel}
\usepackage{geometry}

\usepackage{dsfont}
\usepackage{color}

\usepackage{amsmath,amsfonts,amssymb,amsthm,mathrsfs,dsfont,mathtools}

\usepackage[capitalize,nameinlink]{cleveref}

\usepackage{paralist}
\usepackage[shortlabels]{enumitem}
\usepackage{textcase} % defines '\MakeTextUppercase'
\usepackage{filemod} % '\moddate' returns the last mod date of the file
\usepackage[longnamesfirst]{natbib}
\usepackage{titlesec} % Used to redefine format of section heading
\usepackage{etoolbox} % Use '\AfterEndEnvironment' to avoid indent after theorems

\usepackage[final]{showlabels}

\makeatletter

\def\cite{\citet}

\numberwithin{equation}{section}
\allowdisplaybreaks

\def\@noindentfalse{\global\let\if@noindent\iffalse}
\def\@noindenttrue {\global\let\if@noindent\iftrue}
\def\@aftertheorem{%
  \@noindenttrue
  \everypar{%
    \if@noindent%
      \@noindentfalse\clubpenalty\@M\setbox\z@\lastbox%
    \else%
      \clubpenalty \@clubpenalty\everypar{}%
    \fi}}

\theoremstyle{plain}
\newtheorem{theorem}{Theorem}[section]
\AfterEndEnvironment{theorem}{\@aftertheorem}

\AfterEndEnvironment{definition}{\@aftertheorem}
\newtheorem{lemma}[theorem]{Lemma}
\AfterEndEnvironment{lemma}{\@aftertheorem}
\newtheorem{corollary}[theorem]{Corollary}
\AfterEndEnvironment{corollary}{\@aftertheorem}
\newtheorem{proposition}[theorem]{Proposition}
\AfterEndEnvironment{proposition}{\@aftertheorem}

\theoremstyle{definition}
\newtheorem{remark}[theorem]{Remark}
\AfterEndEnvironment{remark}{\@aftertheorem}

\AfterEndEnvironment{example}{\@aftertheorem}    
\AfterEndEnvironment{proof}{\@aftertheorem}

\titleformat{name=\section}
  {\center\small\sc}{\thesection\kern1em}{0pt}{\MakeTextUppercase}
\titleformat{name=\section, numberless}
  {\center\small\sc}{}{0pt}{\MakeTextUppercase}

\titleformat{name=\subsection}
  {\bf\mathversion{bold}}{\thesubsection\kern1em}{0pt}{}
\titleformat{name=\subsection, numberless}
  {\bf\mathversion{bold}}{}{0pt}{}

\titleformat{name=\subsubsection}
  {\normalsize\itshape}{\thesubsubsection\kern1em}{0pt}{}
\titleformat{name=\subsubsection, numberless}
  {\normalsize\itshape}{}{0pt}{}

\def\note#1{\par\smallskip%
\noindent\kern-0.01\hsize%
\setlength\fboxrule{0pt}\fbox{\setlength\fboxrule{0.5pt}\fbox{%
\llap{$\boldsymbol\Longrightarrow$ }%
\vtop{\hsize=0.98\hsize\parindent=0cm\small\rm #1}%
\rlap{$\enskip\,\boldsymbol\Longleftarrow$}
}%
}%
}

\usepackage{delimset}
\def\given{\mskip 0.5mu plus 0.25mu\vert\mskip 0.5mu plus 0.15mu}
\newcounter{bracketlevel}% 
\def\@bracketfactory#1#2#3#4#5#6{%
\expandafter\def\csname#1\endcsname##1{%
\global\advance\c@bracketlevel 1\relax%
\global\expandafter\let\csname @middummy\alph{bracketlevel}\endcsname\given%
\global\def\given{\mskip#5\csname#4\endcsname\vert\mskip#6}\csname#4l\endcsname#2##1\csname#4r\endcsname#3%
\global\expandafter\let\expandafter\given\csname @middummy\alph{bracketlevel}\endcsname%
\global\advance\c@bracketlevel -1\relax%
}%
}
\def\bracketfactory#1#2#3{%
\@bracketfactory{#1}{#2}{#3}{relax}{0.5mu plus 0.25mu}{0.5mu plus 0.15mu}
\@bracketfactory{b#1}{#2}{#3}{big}{1mu plus 0.25mu minus 0.25mu}{0.6mu plus 0.15mu minus 0.15mu}
\@bracketfactory{bb#1}{#2}{#3}{Big}{2.4mu plus 0.8mu minus 0.8mu}{1.8mu plus 0.6mu minus 0.6mu}
\@bracketfactory{bbb#1}{#2}{#3}{bigg}{3.2mu plus 1mu minus 1mu}{2.4mu plus 0.75mu minus 0.75mu}
\@bracketfactory{bbbb#1}{#2}{#3}{Bigg}{4mu plus 1mu minus 1mu}{3mu plus 0.75mu minus 0.75mu}
}
\bracketfactory{clc}{\lbrace}{\rbrace}
\bracketfactory{clr}{(}{)}
\bracketfactory{cls}{[}{]}
\bracketfactory{abs}{\lvert}{\rvert}
\bracketfactory{norm}{\lVert}{\rVert}
\bracketfactory{floor}{\lfloor}{\rfloor}
\bracketfactory{ceil}{\lceil}{\rceil}
\bracketfactory{angle}{\langle}{\rangle}

\let\original@left\left
\let\original@right\right
\renewcommand{\left}{\mathopen{}\mathclose\bgroup\original@left}
\renewcommand{\right}{\aftergroup\egroup\original@right}

\newcounter{ctr}\loop\stepcounter{ctr}\edef\X{\@Alph\c@ctr}%
	\expandafter\edef\csname s\X\endcsname{\noexpand\mathscr{\X}}
	\expandafter\edef\csname c\X\endcsname{\noexpand\mathcal{\X}}
	\expandafter\edef\csname b\X\endcsname{\noexpand\boldsymbol{\X}}
	\expandafter\edef\csname I\X\endcsname{\noexpand\mathbb{\X}}
\ifnum\thectr<26\repeat

\let\@IE\IE\let\IE\undefined
\newcommand{\IE}{\mathop{{}\@IE}\mathopen{}}
\newcommand{\E}{\mathop{{}\@IE}\mathopen{}}
\let\@IP\IP\let\IP\undefined
\newcommand{\IP}{\mathop{{}\@IP}\mathopen{}}
\renewcommand{\P}{\mathop{{}\@IP}\mathopen{}}
\newcommand{\Var}{\mathop{\mathrm{Var}}}
\newcommand{\Cov}{\mathop{\mathrm{Cov}}}

\newcommand{\R}{\mathbb{R}}

\def\^#1{\relax\ifmmode {\mathaccent"705E #1} \else {\accent94 #1} \fi}
\def\~#1{\relax\ifmmode {\mathaccent"707E #1} \else {\accent"7E #1} \fi}
\def\*#1{\relax#1^\ast}
\edef\-#1{\relax\noexpand\ifmmode {\noexpand\bar{#1}} \noexpand\else \-#1\noexpand\fi}
\def\>#1{\vec{#1}}
\def\.#1{\dot{#1}}

\def\atop{\@@atop}

\renewcommand{\leq}{\leqslant}
\renewcommand{\geq}{\geqslant}
\renewcommand{\phi}{\varphi}

\newcommand\indep{\protect\mathpalette{\protect\@indep}{\perp}}
\def\@indep#1#2{\mathrel{\rlap{$#1#2$}\mkern2mu{#1#2}}}

\def\parsetime#1#2#3#4#5#6{#1#2:#3#4}
\def\parsedate#1:20#2#3#4#5#6#7#8+#9\empty{20#2#3-#4#5-#6#7 \parsetime #8}
\def\moddate{\expandafter\parsedate\pdffilemoddate{\jobname.tex}\empty}

\crefname{equation}{}{}
\crefformat{equation}{\textrm{(#2#1#3)}}
\crefrangeformat{equation}{\mathrm{(#3#1#4)--(#5#2#6)}}

\crefname{page}{p.}{pp.}

\newenvironment{equ}
{\begin{equation} \begin{aligned}}
{\end{aligned} \end{equation}}
\AfterEndEnvironment{equ}{\noindent\ignorespaces}

\newlist{condition}{enumerate}{10}
\setlist[condition]{label*=({A}\arabic*)}
\crefname{conditioni}{condition}{conditions}
\Crefname{conditioni}{Condition}{Conditions}
\def\nin{\not\in}
\DeclareMathOperator{\inj}{inj}
\DeclareMathOperator{\ind}{ind}
\DeclareMathOperator{\Aut}{Aut}

\renewcommand{\leq}{\leqslant}
\renewcommand{\geq}{\geqslant}

\def\func{f}
\def\fung{g}
\newcommand{\mvert}{\,\vert\,}
\renewcommand{\emptyset}{\varnothing}
\renewcommand{\tilde}{\widetilde}
\newcommand{\dd}{\mathop{}\!{d}}

\newcommand{\conep}[2]{\E\left\{ #1 \,\middle\vert\, #2  \right\}}

\makeatother
\usepackage{tikz}
\usetikzlibrary{shapes.geometric}
\tikzset{gon/.style={name=tmp,regular polygon,regular polygon sides=#1,minimum
size=4pt,inner sep=0pt},
polygon side/.style args={#1--#2}{
insert path={(tmp.corner #1)-- (tmp.corner #2)}}}
\newcommand{\FlagGraph}[3][]{\ifnum#2=2%
\tikz[baseline=(tmp1)]{\node[circle,inner sep=0.7pt,fill] (tmp1) at (0,0){};
\node[#1,circle,inner sep=0.7pt,fill] (tmp2) at (5pt,5pt){};
\ifx#3\empty%
\else
\draw[#1] (tmp1) -- (tmp2);
\fi}
\else%
\tikz[baseline=(tmp.south)]{\node[#1,gon=#2]{};
\foreach \X in {1,...,#2}{\fill (tmp.corner \X) circle (1pt);}
\draw[#1,polygon side/.list={#3}]}
\fi}

\def\Gnode{1}

\begin{document}

\title{\sc\bf\large\MakeUppercase{Berry--Esseen bounds for generalized $U$ statistics}}
\author{\sc Zhuo-Song Zhang}

\date{\itshape National University of Singapore}

\maketitle
\providecommand{\keywords}[1]{\textbf{\textit{Keywords:  }}  #1}
\begin{abstract}	
	In this paper, we establish optimal Berry--Esseen bounds for the generalized $U$-statistics. 
	The proof is based on a new Berry--Esseen theorem for exchangeable pair approach by Stein's method under a general linearity condition setting. 
     As applications, an optimal convergence rate of the normal approximation for subgraph counts in Erd\"os--R\'enyi graphs and graphon-random graph is obtained.  \\ 
{\textbf{MSC: } Primnary 60F05; secondary 60K35.}
\\
\end{abstract}

\section{Introduction}

Let $X = (X_1, \dots, X_n) \in \mathcal{X}^n$ and $Y = ( Y_{i,j}, 1 \leq i < j \leq n ) \in \mathcal{Y}^{ n(n-1)/2 }$ be two families of i.i.d.\  random variables; moreover, $X$ and $Y$ are also mutually independent and we set $Y_{j,i} =
Y_{i,j}$ for $j > i$. 
For $k \geq 1$, 
let $\func : \mathcal{X}^k \times \mathcal{Y}^{k(k-1)/2} \to \R$ be a function and we say $\func$ is symmetric if the value of the function $ \func(X_{i_1}, \dots, X_{i_k}; Y_{i_1, i_2}, \dots, Y_{i_{k - 1}, i_k})$ remains unchanged for
any permutation of indices $1 \leq i_1 \neq i_2 \neq \dots \neq i_k \leq n$. 
In this paper, we consider the generalized $U$-statistic defined by
\begin{equ}
	S_{n, k}(\func) = \sum_{\alpha \in \mathcal{I}_{n, k}} \func (X_{\alpha(1)}, \dots, X_{\alpha(k)}; Y_{\alpha(1), \alpha(2)}, \dots, Y_{\alpha(k - 1), \alpha(k)}), 
    \label{eq-Snk}
\end{equ}
where for every $\ell \geq 1$ and $n \geq \ell$,  
\begin{align}
	\label{eq-Inl}
	\mathcal{I}_{n, \ell} = \{ \alpha = (\alpha(1), \dots, \alpha(\ell)) : 1 \leq \alpha(1) < \dots < \alpha(\ell) \leq n \}. 
\end{align}
We note that every $\alpha \in \mathcal{I}_{n,\ell}$ is an $\ell$-fold ordered index.  

As a generalization of the classical $U$-statistic,  
generalized $U$-statistics have been widely applied  in the random graph theory as a count random variable.  
\citet{janson1991asymptotic} studied the limiting behavior of $S_{n,k}(f)$ via a projection method. Specifically, the function $\func$ can be represented as an orthogonal sum of terms indexed by subgraphs of the complete graph with $k$ vertices.
\citet{janson1991asymptotic} showed that the limiting behavior of $S_{n,k}(\func)$ depends on topology of the principle support  graphs (see more details in Subsection 2.1) of $\func$. In particular, the random variable $S_{n,k}(\func)$ is asymptotically normally distributed if
the principle support graphs are all connected. However,
the convergence rate is still unknown. 

The main purpose of this paper is to establish a Berry--Esseen bound for $S_n$ by using Stein's method. 
Stein's method is a powerful tool to estimating convergence rates for distributional approximation. Since introduced by \citet{St72} in \citeyear{St72}, Stein's method has shown to be a powerful tool to evaluate distributional distances for dependent random variables. One of the most important techniques in Stein's method is the exchangeable pair approach, which is commonly taken
in computing the Berry--Esseen bound for both normal and nonnormal approximations. We refer to \citet{St86,rinott1997coupling,ChSh2011} and \citet{Sh16I} for more details on Berry--Esseen bound for bounded exchangeable pairs. It is worth mentioning that \citet{Sh19B}
obtained a Berry--Esseen bound for unbounded exchangeable pairs.

Let $W$ be the random variable of interest, and we say $(W, W')$ is an exchangeable pair if $(W,W') \stackrel{d.}{=} (W',W)$. For normal approximation, it is often to assume the following condition holds: 
\begin{align}
	\mathbb{E} \{ W - W' \vert W \} = \lambda (W + R), 
	\label{eq:linear}
\end{align}
where $\lambda  > 0$ and $R$ is a random variable with a small $ \E | R| $. The condition \cref{eq:linear} can be understood as a linear regression condition. 
Although an exchangeable pair can be easily constructed, it may be not easy to verify the linearity condition
\cref{eq:linear} in some applications. 
% One way to deal with this problem is the embedding method proposed by \citet[]{Re09M}, who obtained a convergence rate for smooth test functions via embedding the target random variable into a higher-dimensional setting. However, it is still open to develop a Berry--Esseen bound
% when the condition \cref{eq:linear} does not hold. 

In this paper, we aim to establish an optimal Berry--Esseen bound for the generalized $U$-statistics by developing a new Berry--Esseen theorem for exchangeable pair approach by assuming a more general condition than \cref{eq:linear}. More
specifically, we replace $W - W'$ in \cref{eq:linear} by a random variable $D$ that is an antisymmetric function of $(X, X')$. 
The new result is given in Section 4.
There are several advantages of our result. Firstly, we propose a new condition more general than \cref{eq:linear} that may be easy to verify. For instance, the condition can be verified by constructing an antisymmetric random variable by the Gibbs sampling method, embedding method,  generalized
perturbative approach and so on. 
Secondly, the Berry--Esseen bound often provides an optimal convergence rate for many practical applications.

The rest of this paper is organized as follows. In Section 2, we give the Berry--Esseen bounds for $S_{n,k}(\func)$. Applications to subgraph counts in $\kappa$-random graphs are given in Section 3. The new Berry--Esseen theorem for exchangeable pair approach under a new setting is established
in Section 4.   We
give the proofs of our main results in Section 5.  
The proofs of other results are postponed to Section 6.

\section{Main results}%
\label{sec:generalized_u_statistics_on_random_graphs}

Let $(X, Y)$, $f$ and $S_{n,k}(f)$ be defined in Section 1. 
For any $\ell \geq 1$, $[\ell] = \{ 1, \dots, \ell\}$ and $[\ell]_2 = \{ (i,j) : 1 \leq i,j \leq \ell\}$.  
Let $A \subset [\ell]$ and let $B \subset [\ell]_2$, and let $X_{A} = (X_i : i \in A)$ and $Y_B = (Y_{i,j} : (i,j) \in B)$. 
Specially, we can simply write $f(X_1, \dots, X_{k}; Y_{1,2}, \dots, Y_{k - 1,k})$ as $f(X_{[k]}; Y_{[k]_2})$.
Let $G_{A,B}$ be the graph with vertex set $A$ and edge set $B$, and let $v_{A,B}$ be the number of nodes in $G_{A,B}$.

By the Hoeffding decomposition, we have 
\begin{align*}
	f(X_{[k]}; Y_{[k]_2})
	= \sum_{A \subset [k] , B \subset [k]_2} f_{A,B}(X_{A}; Y_{B}), 
\end{align*}
where $f_{A,B}: \mathcal{X}^{|A|} \times \mathcal{Y}^{|B|} \to \R$ is defined as 
\begin{multline}
	f_{A,B}(x_{A}; y_{B}) = \sum_{ (A', B') : A' \subset A , B' \subset B } (-1)^{|A| + |B| - |A'| - |B'|} \\
	\times \E \bigl\{ f( X_1, \dots, X_k; Y_{1,2}, \dots, Y_{k - 1, k} ) \bigm\vert X_{A'} = x_{A'}, Y_{B'} = y_{B'} \bigr\}, 
	\label{eq-fAB}
\end{multline}
where $x_{A} = \{ x_i : i \in A \}$ and $y_{B} = \{ y_{i,j} : (i,j) \in B \}$ for $A \subset [k]$ and $B \subset [k]_2$.
We remark that if $A = \emptyset$ and $B = \emptyset$, then $f_{\emptyset,\emptyset}(X_{\emptyset}; Y_{\emptyset}) = \E \{ f(X_{[k]}; Y_{[k]_2}) \}$. 
For $\ell = 0, 1, \dots, k$, let 
\begin{equ}
	f_{(\ell)}(X_{[k]}; Y_{[k]_2}) = 
	\begin{dcases}
		\E \{f(X_{[k]}; Y_{[k]_2})\} & \text{ if $\ell = 0$, }\\
		\sum_{ v_{A,B}= \ell } f (X_A; Y_B) & \text{ if $\ell \geq 1$}, 
	\end{dcases}
    \label{eq-psiell}
\end{equ}
where $v_{A,B}$ is the number of nodes in $G_{A,B}$.
Let $d = \min\{ \ell > 0 : f_{(\ell)} \neq 0\}$, and we call $d$ the \emph{principal degree} of $f$. We say $f_{(d)}$ is the \emph{principal part} of $f$. Moreover, we say the subgraphs $G_{A,B}$ such that $v_{A,B} = d$ and $f_{A,B} \neq 0$ are the \emph{principal support graphs} of $f$.

The central limit theorems for $S_{n,k}(f)$ is proved by \citet{janson1991asymptotic}. 
Let $\sigma_{A,B} = \lVert f_{A, B} (X_A; Y_B) \rVert$, and let $\mathcal{G}_{f,d} = \{G_{A,B}:  \sigma_{A,B} \neq 0, v_{A,B} = d\}$ be the set of principal index graph. 
We remark that if $f$ has the principal degree $d$, then $\Var (S_{n,k}(f))$ is of order $n^{2k-d}$, see Lemmas 2 and 3 in \citet{janson1991asymptotic}.
\citet{janson1991asymptotic} proved that if all graphs in $\mathcal{G}_f$ are connected, then 
\begin{align*}
	\frac{ S_{n,k}(\func) - \E \{ S_{n,k}(\func) \} }{ (\Var ( S_{n,k}(\func) ))^{1/2} } \stackrel{d.}{\to} N(0, 1). 
\end{align*}
Note that if not all principal support graphs are connected, then the limiting distribution of the scaled version of $S_{n,k}$ is nonnormal (see Theorems 2 and 3 in \citet{janson1991asymptotic}), and we will consider this case in another paper.  

% For any $m \in \mathcal{I}_{n, k}$, write $i \in m$ if $i \in \{ m_1, \dots, m_k \}$; moreover, let $\tilde{X}_m = (X_{m_1}, \dots, X_{m_k})$ and $\tilde{Y}_m = (Y_{m_1,m_2}, \dots, Y_{m_{k - 1}, m_k})$. Assume that $\E \{ \func(\tilde{X}_m; \tilde{Y}_m) \} = 0$ and let  
% \begin{align*}
% 	g(x) & =  \E \{ \func( \tilde{X}_m; \tilde{Y}_m ) \mvert X_{m_1} = x \}. 
% 	% h(x_1, x_2 ; y) & =  \E \{ \func( \tilde{X}_m; \tilde{Y}_m ) \mvert X_{m_1} = x_1, X_{m_2} = x_2, Y_{m_1, m_{2}} = y \}.
% \end{align*}

Now, assume that $\func$ is a symmetric function having principal degree $d$ ($1 \leq d \leq k$).
In this subsection, we give a Berry--Esseen bound for $S_{n,k}(\func)$. 
For $x \in \mathcal{X}$, let  
\begin{align*}
	f_{\Gnode}(x) \coloneqq f_{\{1\}, \emptyset}(x) =  \E \{\func(X_{[k]}; Y_{[k]_2}) \mvert X_1 = x\} - \E \{ \func ( X_{[k]}; Y_{[k]_2} ) \}. 
\end{align*}
If $\lVert f_{\Gnode}(X_1) \rVert_2 > 0$, then it follows that $d = 1$. 
Here and in the sequel, we denote by $\|Z\|_p \coloneqq (\E |Z|^p)^{1/p}$ for $p > 0$ and we denote by $\Phi(\cdot)$ the distribution function of $N(0,1)$.   
The following theorem provides the Berry--Esseen bound for $S_{n,k}(\func)$ in the case where $\lVert f_{\Gnode}(X_1) \rVert_2 > 0$.
\begin{theorem}
	\label{thm-3.1}
	% Let $\tau = \lVert \func( X_{[k]}; Y_{[k]_2} ) \rVert_4 < \infty$. 
	If $\sigma_{\Gnode} \coloneqq \lVert f_{\Gnode}(X_1) \rVert_2 > 0$, then 
	\begin{equ}
		\sup_{z \in \R} \biggl\lvert \P \biggl[ \frac{S_{n,k}(f) - \E \{S_{n,k}(f)\}}{ \sqrt{\Var \{ S_{n,k}(f) \}} }  \leq z \biggr] - \Phi(z) \biggr\rvert  \leq  \frac{12k \lVert f(X_{[k]}; Y_{[k]_2}) \rVert_4^{2}}{\sqrt{n}  \sigma_{\Gnode}^2} .
		\label{t3.1-ab}
	\end{equ}
\end{theorem}
\begin{remark}
	We remark that $\Var ( S_{n,k}(f) ) = O(n^{2k-1})$ as $n \to \infty$. Typically, the right hand side of \cref{t3.1-ab} is of order $n^{-1/2}$. Specially, if $f(X_{[k]},Y_{[k]_2}) = h(X_{[k]})$ for some symmetric function $h : \mathcal{X}^k \to \R$, then $S_{n,k}$ is the classical $U$-statistic. In this case, \citet{Ch07N} obtained a
	Berry--Esseen bound of order $n^{-1/2}$ under the assumption that $\lVert h(X_{[k]}) \rVert_3 < \infty$.
\end{remark}

If $\sigma_{\Gnode} = 0$, then $d \geq 2$, that is, the principal degree of $f$ is at least $2$. 
We have the following theorem. 
\begin{theorem}
	\label{thm-3.0}
	Let $\tau \coloneqq  \lVert \func(X_{[k]}; Y_{[k]_2}) \rVert_4 < \infty$ and let $\sigma_{\min} \coloneqq \min ( \sigma_{A,B} : G_{A,B} \in \mathcal{G}_{f,d} )$. Assume that $\func$ is a symmetric function having principal degree $d$ for some $2 \leq d \leq k$, and assume further
	that for all  graphs in $\mathcal{G}_{f,d}$ are connected. Then, we have 
	\begin{align*}
		\sup_{z \in \R} \biggl\lvert \P \biggl[  \frac{(S_{n,k}(\func) - \E \{ S_{n,k}(\func) \} )}{\sqrt{\Var \{ S_{n,k}( \func ) \}}} \leq z  \biggr] - \Phi(z) \biggr\rvert \leq C n^{-1/2}, 
	\end{align*}
	where  $C > 0$ is a constant depending only on $k, d, \sigma_{\min}$, and $\tau$. 
\end{theorem}

If we further assume that the function $\func$ does not depend on $X$, i.e., $\func(X; Y) = g(Y)$ for some symmetric $g : \mathcal{Y}^{k(k-1)/2} \to \R$,   we obtain a sharper convergence rate. 
To give the theorem, we first introduce some more notation. Let $G^{(r)}$ be the graph generated from $G$ by deleting the node $r$ and all the edges connecting to the node $r$. We say $G$ is \emph{strongly connected} if 
$G^{(r)}$ is connected or empty for all $r \in V(G)$. We note that all strongly connected graphs are also connected. 
The following theorem provides a sharper Berry--Esseen bound than that in \cref{thm-3.0}.

\begin{theorem}
	\label{thm-3.2}
	Assume that $f(X_{[k]}; Y_{[k]_2}) = g(Y_{[k]_2})$ almost surely for some symmetric $g:\mathcal{Y}^{k(k-1)/2} \to \R$. Let $\tau$ and $\sigma_{A,B}$ be defined in \cref{thm-3.0}. Assume that the conditions in \cref{thm-3.0} are satisfied and assume further that all graphs in $\mathcal{G}_{f,d}$ are strongly connected.  Then,
	\begin{align*}
		\sup_{z \in \R} \biggl\lvert \P \biggl[  \frac{(S_{n,k}(g) - \E \{ S_{n,k}(g) \} )}{\sqrt{\Var \{ S_{n,k}( g ) \}}} \leq z  \biggr] - \Phi(z) \biggr\rvert \leq C n^{-1}, 		 
	\end{align*}
	where $C > 0$ is a constant depending on $k, d, \sigma_{\min}$, and $\tau$. 
\end{theorem}

\section{Applications}%
\label{sec:applications}

\subsection{Subgraphs counts in random graphs generated from graphons}%
\label{sub:subgraphs_counts_in_graphon_generated_random_graphs}

A symmetric Lebesgue measurable function $\kappa : [0,1]^2 \to [0,1]$ is called a \emph{graphon}, which was firstly introduced by \cite[]{lovasz2006} to represent the graph limit. 
Given a graphon $\kappa$ and $n \geq 2$, the $\kappa$-random graph $\mathbb{G}(n, \kappa)$ can be generated as follows: 
Let $n \geq 1$ and let $X = (X_1, \dots, X_n)$ be a vector of independent uniformly distributed
random variables on $[0,1]$. Given $X$, we generate the graph $\mathbb{G}(n, \kappa)$ by connecting the node pair $(i,j)$ independently with probability $\kappa(X_i, X_j)$. 
This construction was firstly introduced by \citet{Dia81}, which can be used to study large dense and sparse random graphs and random trees generated from graphons. We refer to \citet{lovasz2006,Bol07,lovasz2012large} for more details. 

\def\ER{\mathrm{ER}}
Subgraph counts are important statistics in estimating graphons. As a special case, when $\kappa \equiv p$ for some $p \in (0,1)$, the $\kappa$-random graph model becomes the classical Erd\"os--R\'enyi model $\ER(p)$. The study of asymptotic properties of subgraph counts in
$\ER(p)$ dates back to \citet{Now89,barbour1989central,janson1991asymptotic} for more details. Recently, \citet{Kr17D}, \citet{Rol17} and \citet{Pri18a}  applied Stein's method to obtain an optimal Berry--Esseen bound for triangle counts in $\ER(p)$. For subgraph counts in
$\kappa$-random graph, \citet{Kau20a} proved an upper bound of the Kolmogorov distance for multivariate normal approximations for centered subgraph counts with order $n^{-1/(p + 2)}$ for some $p > 0$.
However, the Berry--Esseen bounds for subgraph counts of $\kappa$-random graph is still unknown so far. 
In this subsection, we apply \cref{thm-3.0,thm-3.2} to prove sharp Berry--Esseen bounds for
subgraph counts statistics.

Let $\Xi = (\xi_{i,j})_{1 \leq i < j \leq n}$ be the adjacency matrix of $\mathbb{G}(n, \kappa)$, where for each $(i,j)$,
the binary random variable $\xi_{i,j}$ indicates the
connection of the graph. 
Formally, let $Y = (Y_{1,1}, \dots, Y_{n-1,n})$ be a vector of independent uniformly distributed random variables that is also independent of $X$, and then we can write $\xi_{i,j} = \IN{Y_{i,j} \leq \kappa(X_i, X_j)} $.
For any nonrandom simple $F$ with $ v(F)  = k$, the (injective) subgraph counts and induced subgraph counts in $ \mathbb{G}(n, \kappa) $ are defined by  
\begin{align*}
	T_F^{\inj} \coloneqq T_F^{\inj} ( \mathbb{G}(n, \kappa) ) & = \sum_{\alpha \in \mathcal{I}_{n,k}} \phi_F^{\inj} (\xi_{\alpha(1),\alpha(2)}, \dots, \xi_{\alpha({k-1}), \alpha({k})}) ,
	\\
	T_F^{\ind} \coloneqq T_F^{\ind}( \mathbb{G}(n, \kappa) ) & =  \sum_{\alpha \in \mathcal{I}_{n,k}} \phi_F^{\ind} (\xi_{\alpha(1), \alpha(2)}, \dots, \xi_{\alpha({k - 1}), \alpha (k)}), 
\end{align*}
respectively, where for $(x_{1,1}, \dots, x_{k - 1, k}) \in \mathbb{R}^{k(k - 1)/2}$, 
\begin{align*}
	\phi_F^{\inj} (x_{1,1}, \dots, x_{k-1,k})  & =  \sum_{H: H \cong F} \prod_{(i,j) \in E(H)} x_{i,j}
	, \\
	\phi_F^{\ind} ( x_{1, 2}, \dots, x_{{k - 1}, k}  )
												 & =  \sum_{H: H \cong F} \prod_{(i,j) \in E(H)} x_{i,j}\prod_{(i,j) \not\in E(H)} (1 - x_{i,j}).
	% & = 
	% \begin{dcases}
	% 	1 & \text{if the subgraph of $G$ induced by} \\
	% 	\phantom{ab} & \text{$1, \dots, k$ is isomorphic to $F$}, \\
	% 	0 & \text{otherwise},
	% \end{dcases}
\end{align*}
Here, the summation $\sum_{H: H \cong F}$ ranges over the subgraphs with $v(F)$ nodes that are isomorphic to $F$ and thus contains $v(F)!/|\Aut(F)|$ terms, where $\lvert \Aut(F) \rvert$ is the number of automorphisms of $F$.  Moreover, we note that both $\phi_F^{\inj}$ and $\phi_F^{\ind}$ are symmetric.
For example, if $F$ is the $2$-star, then $k = 3$, $|\Aut(F)| = 2$ and 
\begin{align*}
	\phi_F^{\inj}( \xi_{1,2}, \xi_{1,3}, \xi_{2,3} ) & =  \xi_{1,2} \xi_{1,3} + \xi_{1,2} \xi_{2,3}+ \xi_{1,3} \xi_{2,3} ,\\
	\phi_F^{\ind}( \xi_{1,2}, \xi_{1,3}, \xi_{2,3} ) & = \xi_{1,2} \xi_{1,3} (1 - \xi_{2,3}) + \xi_{1,2} \xi_{2,3} (1 - \xi_{1,3}) + \xi_{1,3} \xi_{2,3} (1 - \xi_{1,2}).
\end{align*}
If $F$ is a triangle, then $|\Aut(F)| = 6$ and 
\begin{align*}
	\phi_F^{\inj}(\xi_{1,2}, \xi_{1,3}, \xi_{2,3})= \phi_F^{\ind}(\xi_{1,2}, \xi_{1,3},\xi_{2,3}) = \xi_{1,2}\xi_{1,3}\xi_{2,3}.
\end{align*}

Let 
\begin{align*}
	t_F(\kappa) & =  \int_{ [0,1]^{k} } \prod_{(i,j) \in E(F)} \kappa(x_i, x_j)  \prod_{i \in V(F)} d x_i,\\  
	t_F^{\ind}(\kappa) & =  \int_{ [0,1]^{k} } \prod_{(i,j) \in E(F)} \kappa(x_i, x_j) \prod_{(i,j) \nin E(F)} (1 - \kappa(x_i, x_j)) \prod_{i \in V(F)} d x_i.  
\end{align*}
Then, we have 
\begin{align*}
	\E \{ \phi_F^{\inj} (\xi_{1,1}, \dots, \xi_{k-1,k})\}
	& = \frac{k!}{|\Aut(F)|} t_F(\kappa), \\
	\E \{ \phi_F^{\ind} (\xi_{1,1}, \dots, \xi_{k-1,k})\}
	& = \frac{k!}{|\Aut(F)|} t_F^{\ind}(\kappa). 
\end{align*}

% Specially, if $\kappa \equiv p$ for some $0 < p < 1$, then $\mathbb{G}(n, \kappa)$ is the Erd\"os--R\'enyi model. In this special case, the normal approximation has been proved by \citet{Now89} using the technique of incomplete $U$ statistics. Moreover,
% \citet{barbour1989central} provides a $L_1$ bound using Stein's method. For the special case where $F$ is a triangle, \citet{Rol17} obtained Berry--Esseen bound   using a combination of Stein's method and characteristic functions. 

As $\xi_{i,j} = \IN{ Y_{i,j} \leq \kappa(X_i, X_j) }$, let 
\begin{align*}
	\func_F^{\inj} (X_{[k]}; Y_{[k]_2}) & = \phi_F^{\inj} ( \xi_{1,1}, \dots, \xi_{k - 1, k} ), 
\end{align*}
Now, as random variables $(\xi_{i,j})_{1 \leq i < j \leq n}$ are conditionally independent given $X$, we have 
\begin{align*}
	\E \{ \func_F^{\inj} ( X_{[k]}; Y_{[k]_2} ) \mvert X \} 
	& = \sum_{H \cong F} \prod_{(i,j) \in E(H)} \kappa (X_i, X_j),  \\
	\E \{ \func_F^{\ind} ( X_{[k]}; Y_{[k]_2} ) \mvert X \} 
	& = \sum_{H \cong F} \prod_{(i,j) \in E(H)} \kappa (X_i, X_j) \prod_{(i,j) \nin E(H)} (1 - \kappa(X_i, X_j)).  
\end{align*}
Let 
\begin{align*}
	f_{\Gnode}^{\inj}(x) 
	& = \E \{ \func_F^{\inj}	( X_{[k]}; Y_{[k]_2} ) \mvert X_1 = x \} \\
	& = \sum_{H \cong F} \E \biggl\{ \prod_{(i,j) \in E(H)} \kappa(X_i, X_j) \biggm\vert X_1 = x\biggr\}	, 
\end{align*}
and similarly, let 
\begin{align*}
	f_{\Gnode}^{\ind}(x) 
	& = 
	\E \{ \func_{F}^{\ind} ( X_{[k]}; Y_{[k]_2} ) \mvert X_1 = x \}\\
	 & =  \sum_{H \cong F} \E \biggl\{ \prod_{(i,j) \in E(H)} \kappa(X_i, X_j) \prod_{(i,j) \nin E(H)} (1 - \kappa(X_i, X_j)) \biggm\vert X_1 = x \biggr\}. 
\end{align*}
We have the following theorem, which follows from \cref{thm-3.1} directly. 
\begin{theorem}
	Let $\sigma_{\Gnode}^{\inj} = \lVert f_{\Gnode}^{\inj}(X_1) - \E \{ f_{\Gnode}^{\inj}(X_1) \} \rVert_2$ and $\sigma_{\Gnode}^{\ind} = \lVert f_{\Gnode}^{\ind}(X_1) - \E \{g_F^{\ind}(X_1)\} \rVert_2$. Assume that $\sigma_{\Gnode}^{\inj}> 0$, then 
	\begin{align*}
		\sup_{z \in \R} \biggl\lvert \P \biggl[ \frac{\sqrt{n} }{k \sigma_{\Gnode}^{\inj}} \binom{n}{k}^{-1} (T_F^{\inj} - \E \{ T_F^{\inj} \}) \leq z \biggr] - \Phi(z) \biggr\rvert
		& \leq C n^{-1/2}. 
	\end{align*}
	Moreover, assume that $\sigma_{\Gnode}^{\ind} > 0$, then
	\begin{align*}
		\sup_{z \in \R} \biggl\lvert \P \biggl[ \frac{\sqrt{n} }{k \sigma_{\Gnode}^{\ind}} \binom{n}{k}^{-1} (T_F^{\ind} - \E \{T_F^{\ind}\}) \leq z \biggr] - \Phi(z) \biggr\rvert
		& \leq C n^{-1/2}.
	\end{align*}
\end{theorem}

If $\kappa \equiv p$ for a fixed number $0 < p < 1$, then the random variables $( \xi_{i,j} )_{1 \leq i < j \leq n}$ are i.i.d.\  and the functions $\phi_F^{\inj}$ and $\phi_F^{\ind}$ do not depend on $X$. We have the following theorem: 
\begin{theorem}
	\label{thm3.2}
	Let $\kappa \equiv p$ for $0 < p < 1$. Then
	\begin{align*}
		\sup_{z \in \R} \biggl\lvert \P \biggl[ \frac{ T_F^{\inj} - \E \{ T_F^{\inj} \} }{ ( \Var \{ T_F^{\inj} \} )^{1/2}  } \leq z \biggr] - \Phi(z) \biggr\rvert
		& \leq C n^{-1}. 
	\end{align*}
\end{theorem}
\begin{remark}
	For the $L_1$ bound, \citet{barbour1989central} proved the same order of $O(n^{-1})$ in the case that $p$ is a constant. 
	For the Berry--Esseen bound, \citet{Pri18a} proved a general Berry--Esseen bound for subgraph counts for Erd\"os--R\'enyi random graph using a different method. Specially, if $p$ is a constant, then \cref{thm3.2} provides the same result as in \citet{Pri18a}. 
\end{remark}

For induced subgraph counts, we need to consider some separate cases. Let $s(F)$ and $t(F)$ denote the number of 2-stars and triangles in $F$, respectively. If any of the following conditions holds, then it has been proven by \cite{janson1991asymptotic} that $ ( T_F^{\ind} - \E \{ T_F^{\ind} \} )/(  \Var \{ T_F^{\ind} \}
)^{1/2} $
	converges to a standard normal distribution: 
	\begin{enumerate}[({G}1)]
	\item If $e( F ) \neq p \binom{v(F)}{2}$; 
	\item if $e( F ) = p \binom{v(F)}{2}$, $s(F) \neq 3 p^2 \binom{v(F)}{3}$; 
	\item if $ e( F ) = p \binom{v(F)}{2} $, $s(F) = 3 p^2 \binom{v(F)}{3}$ and  $t(F) \neq p^3 \binom{v(F)}{3}$. 
\end{enumerate}

The following theorem gives the Berry--Esseen bounds for induced subgraph counts. 
\begin{theorem}
	Let $\kappa \equiv p$ for $0 < p < 1$. If (G1) or(G3) holds, then 
	\begin{equ}
		\sup_{z \in \R} \biggl\lvert \P \biggl[ \frac{ T_F^{\ind} - \E \{ T_F^{\ind} \} }{ ( \Var \{ T_F^{\ind} \} )^{1/2}  } \leq z \biggr] - \Phi(z) \biggr\rvert
		& \leq C n^{-1}.
        \label{eq-3.3-1}
	\end{equ}
If (G2) holds, then 
\begin{equ}
		\sup_{z \in \R} \biggl\lvert \P \biggl[ \frac{ T_F^{\ind} - \E \{ T_F^{\ind} \} }{ ( \Var \{ T_F^{\ind} \} )^{1/2}  } \leq z \biggr] - \Phi(z) \biggr\rvert
		& \leq C n^{-1/2}.
        \label{eq-3.3-2}
	\end{equ}
\end{theorem}

\section{A new Berry--Esseen bound for exchangeable pair approach}
% \subsection{Exchangeable pair approach}%
% \label{sub:epa}
\subsection{Berry--Esseen bound}%
\label{sub:berry_esseen_bound}
In this section, we establish a new Berry--Esseen theorem for exchangeable pair approach under a new setting. 
Let $X \in \mathcal{X}$ be a random variable valued on a measurable space and let $W = \phi(X)$ be the random variable of interest where $\phi: \mathcal{X} \to \R$. Assume that $\E \{ W \} = 0$ and $\E \{ W^2 \} = 1$. We propose the following condition: 
\begin{enumerate}
	\item [(A)] Let $(X, X')$ be an exchangeable pair and let $F : \mathcal{X} \times \mathcal{X} \to \R$ be an antisymmetric function. Assume that $D \coloneqq F(X, X')$ satisfies the following condition: 
		\begin{equ}
			\E \{ D \vert X \} = \lambda ( W + R ),	
            \label{eq00}
		\end{equ}
		where $\lambda > 0$ is a constant and $R$ is a random variable. 
\end{enumerate}
We remark that the operator of antisymmetric functions was firstly mentioned by \citet{Hol04}, and the condition (A) was considered by \citet{Cha07b}, who applied the exchangeable pair approach to prove concentration inequalities.

The following theorem provides a uniform Berry--Esseen bound for exchangeable pair approach under the assumption (A). 
\begin{theorem}
	\label{thm01}
	Let $(X, X')$ and $D$ satisfy the condition \textup{(A)}. Let $W' = \phi(X')$ and $\Delta = W - W'$. Then, 
	\begin{align}
		\sup_{z \in \R} \lvert \P [ W \leq z ] - \Phi(z) \rvert 
		& \leq \E \biggl\lvert 1 - \frac{1}{2\lambda} \E \{ D \Delta \mvert W \} \biggr\rvert + \frac{1}{\lambda} \E \bigl\lvert \E \{ D^* \Delta \mvert W \} \bigr\rvert + \E \lvert R \rvert, 
		\label{eq:thm1}
	\end{align}
	provided that $D^* \coloneqq F^* (X, X') \geq |D|$, where $F^*$ is a symmetric function.
\end{theorem}

\begin{remark}
	Assume that \cref{eq:linear} is satisfied. Then, we can choose $D = \Delta = W - W'$, and the right hand side of \cref{eq:thm1} reduces to 
	\begin{align*}
		\E \biggl\lvert 1 - \frac{1}{2\lambda} \E \{ \Delta^2 \mvert W \} \biggr\rvert + 
		\frac{1}{\lambda} \E \bigl\lvert \E \{ \Delta^* \Delta  \} \bigr\rvert + \E |R|,
	\end{align*}
	where $\Delta^* \coloneqq \Delta^* (W, W')$ is a symmetric function for $W$ and $W'$ such that $\Delta^* \geq |\Delta|$. Thus, \cref{thm01} recovers to Theorem 2.1 in \citet{Sh19B}. 
\end{remark}

% The following corollary might be useful for bounded random variables. 
% \begin{corollary}
% 	Under the assumptions of \cref{thm01}. Assume that there exists a constant $\delta > 0$ and a function $g : \R \to \R$ such that $|D| \leq \delta$ and $\E \{ \Delta \mvert W \} = \lambda g(W)$, then 
% 	\begin{align*}
% 		\sup_{z \in \R} \lvert \P[ W \leq z ] - \Phi(z) \rvert 
% 		& \leq \E \biggl\lvert 1 - \frac{1}{2\lambda} \E \{ D \Delta \mvert W \} \biggr\rvert + \delta \E \bigl\lvert g(W) \bigr\rvert + \E \lvert R \rvert. 
% 	\end{align*}
% \end{corollary}

The following corollary is useful for random variables that can be decomposed as a sum of $W$ and a remainder term. Specifically, 
let $T \coloneqq T(X)$ be a random variable such that $T = W + U$, where $W = \phi(X)$ is as defined at the beginning of this section, and $U \coloneqq U(X)$ is a remainder term. The following corollary gives a Berry--Esseen bound for $T$. 
\begin{corollary}
	\label{cor-nonlinear}
	Let $(X, X') \in \mathcal{X} \times \mathcal{X}$ be an exchangeable pair and let $D \coloneqq F(X, X')$ where $F: \mathcal{X} \times \mathcal{X} \to \R$ is antisymmetric. Assume that  
	\begin{align}
		\E \{ D \mvert X \} = \lambda (W + R)
		\label{eq-corcon}
	\end{align}
	for some $\lambda > 0$ and some random variable $R$. Let $U' \coloneqq U(X')$ and $\Delta = \phi(X) - \phi(X')$. Then, we have 
	\begin{multline*}
		\sup_{z \in \R} \bigl\lvert \P [T \leq z] - \Phi(z) \bigr\rvert
		\leq \E \biggl\lvert 1 - \frac{1}{2\lambda} \E \{ D \Delta \mvert X \} \biggr\rvert \\
		+ \frac{1}{\lambda} \E \bigl\lvert \E \{ D^* \Delta \mvert X \} \bigr\rvert + \frac{3}{2\lambda} \E \lvert D(U - U') \rvert + \E |R| + \E |U|, 
	\end{multline*}
	provided that $D^* \coloneqq D^* (X, X')$ is any symmetric function of $X$ and $X'$ such that $D^* \geq |D|$.
\end{corollary}
\begin{remark}
	Assume that $X = (X_1, \dots, X_n) $ is a family of independent random variables. Let $W = \sum_{i = 1}^n \xi_i$ be a linear statistic, where $\xi_i = h_i(X_i)$ and $h_i$ is a nonrandom function, such that $\E \{ \xi_i \} = 0$ and $\sum_{i = 1}^n \E \{ \xi_i^2 \} = 1$, and let
	$U = U(X_1, \dots, X_n) \in \R$ be a random variable. Let $T = W + U$, $\beta_2 = \sum_{i = 1}^n \E \{ \lvert \xi_i \rvert^2 \IN { \lvert \xi_i \rvert > 1 } \}$ and $\beta_3 = \sum_{i = 1}^n \E \{  \lvert \xi_i\rvert^3 \IN { \lvert \xi_i \rvert \leq 1 }
	\}$.
	\citet{Ch07N} (see also \citet{Sha16}) proved the following result: 
	\begin{align}
		\sup_{z \in \R} \lvert \P [ T \leq z ] - \Phi(z) \rvert 
		& \leq 17 (\beta_2 + \beta_3) + 5 \E |U| + 2 \sum_{i = 1}^n \E \lvert \xi_i(U - U^{(i)}) \rvert, 
		\label{eq-shzh16}
	\end{align}
	where $U^{(i)}$ is any random variable independent of $\xi_i$. 

	The Berry--Esseen bound in \cref{cor-nonlinear} improves \citet{Ch07N}'s result in the sense that the random variable $W$ in our result is not necessarily a partial sum of independent random variables, and our result in \cref{cor-nonlinear} can be applied to a general class of random variables. 
	% Specially, under the setting of \citet{Ch07N}, applying \cref{cor-nonlinear} yields a Berry--Esseen bound for $T$ with the same order as \cref{eq-shzh16}:  
	% \begin{equ}
	% 	\sup_{ z \in \R } \lvert \P [ T \leq z ] - \Phi(z) \rvert 
	% 	& \leq 4.5 \bigl( \beta_2 + \beta_4^{1/2} \bigr) + 3 \sum_{i = 1}^n \E \lvert \xi_i (R - R^{(i)}) \rvert + \E |R|,
        % \label{eq-nonlin2}
	% \end{equ}
	% where $\beta_4 = \sum_{i = 1}^n \E \{\xi_i^4 \IN{ \lvert\xi_i\rvert \leq 1 }\}$.

\end{remark}
\subsection{Proof of Theorem 4.1}%
\label{sub:proof_of_theorem_31}
In this subsection, we prove Theorem 4.1 by Stein's method. The proof is similar to that of Theorem 2.1 in \citet{Sh19B}.  To begin with, we need to prove the following lemma, which is useful in the proof of Theorem 4.1. 
\begin{lemma}
    Let $\func$ be a nondecreasing function. Then, 
    \begin{align*}
        \frac{1}{2\lambda} \biggl\lvert \E \biggl\{ D \int_{-\Delta}^0 \bigl(  \func(W + u) - \func(W) \bigr) \dd u  \biggr\}\biggr\rvert \leq \frac{1}{2\lambda} \E \bigl\{ D^{*} \Delta \func(W) \bigr\}, 
    \end{align*}
    where $D^*$ is as defined in \cref{thm01}. 
    \label{lem:psi}
\end{lemma}
\begin{proof}
    [Proof of \cref{lem:psi}] 
    Since $\func(\cdot)$ is nondecreasing, it follows that
	$$\Delta \bigl( \func(W) - \func(W') \bigr) \geq 0$$ 
	and 
    \begin{align*}
        0 & \geq \int_{-\Delta}^0 \bigl( \func(W + u) - \func(W) \bigr) \dd u \\
          & \geq - \Delta \bigl( \func(W) - \func(W') \bigr), 
    \end{align*}
    which yields 
    \begin{align*}
        - \E \bigl\{ D \IN{ D > 0 } \Delta \bigl( \func(W) - \func(W') \bigr) \bigr\} 
        & \leq \E \biggl\{ D \int_{-\Delta}^0 \bigl( \func(W + u) - \func(W) \bigr) \dd u \biggr\}\\
        & \leq - \E \bigl\{ D \IN{ D < 0 } \Delta \bigl( \func(W) - \func(W') \bigr) \bigr\}.  
    \end{align*}
    Recalling that $ W= \phi({X})$, $D = F({X}, {X}')$ is antisymmetric and $D^* = F^*({X}, {X}')$ is symmetric, as $({X}, {X}')$ is exchangeable,  we have 
    \begin{gather*}
        \E \bigl\{ D \IN{ D > 0 } \Delta \bigl\{ \func(W) - \func(W') \bigr\} \bigr\} = - \E \bigl\{ D \IN{ D < 0 } \Delta \bigl( \func(W) - \func(W') \bigr) \bigr\}, 
        \intertext{and}
        \E \bigl\{ D^* \IN{ D > 0 } \Delta \func(W) \bigr\} = - \E \bigl\{ D^* \IN{ D < 0 } \Delta f (W') \bigr\}.
    \end{gather*}
    Moreover, 
    as  $\E \bigl\{D^* \Delta \IN{ D = 0} \bigl( \func(W) - \func(W') \bigr)\bigr\} \geq 0$ and $\E \bigl\{ D^* \IN{ D = 0 } \Delta \func(W) \bigr\} = - \E \bigl\{ D^* \IN{ D = 0 } \Delta \varphi(W') \bigr\}$, it follows that 
    \begin{align*}
        \E \{D^* \Delta \IN{ D = 0} \func(W)\} \geq 0. 
    \end{align*}
    Therefore, 
    \begin{align*}
        \MoveEqLeft \frac{1}{2\lambda} \biggl\lvert \E \biggl\{ D \int_{-\Delta}^0 \bigl\{ \func(W + u)  - \func(W) \bigr\} \dd u\biggr\}\biggr\rvert \\
        & \leq - \frac{1}{2\lambda} \E \bigl\{D \IN{ D < 0 } \Delta \bigl( \func(W) - \func(W') \bigr)\bigr\} \\
        & \leq \frac{1}{2\lambda} \E \bigl\{D^* \IN{ D < 0 } \Delta \bigl( \func(W) - \func(W') \bigr) \bigr\}\\
        %& \leq \frac{1}{2\lambda}  \E D^* \bigl( \IN{ D > 0} + \IN{ D < 0 }  \bigr) \Delta \func (W) \\ 
        & = \frac{1}{2\lambda} \E \bigl\{D^{*} \Delta \bigl( \IN{ D > 0 } + \IN{ D <0 } \bigr) \func(W)\bigr\} \\
        & \leq \frac{1}{2\lambda} \E \{D^{*} \Delta \func(W)\}.  \qedhere
    \end{align*}
\end{proof}
\begin{proof}
    [Proof of \cref{thm01}] 
	We apply some ideas of Theorem 2.1 in \citet{Sh19B} to prove the desired result. 
    Let $z \geq 0$ be a fixed real number,  
    and $f_z$ the solution to the Stein equation: 
    \begin{align}
        f'(w) - w f(w) = \IN{  w \leq z  } - \Phi(z), 
        \label{eq:stein}
    \end{align}
    where $\Phi(\cdot)$ is the distribution function of the standard normal distribution. 
    It is well known that (see, e.g., \citet{Ch11N0})
    \begin{equ}
        \label{eq:solu}
        f_z (w) = 
        \begin{dcases}
            \sqrt{2\pi} e^{w^2/2} {\Phi(w) \bigl\{ 1 - \Phi(z) \bigr\}} & \text{if } w \leq z, \\
            \sqrt{2\pi} e^{w^2/2}{ \Phi(z) \bigl\{ 1 - \Phi(w) \bigr\} } & \text{otherwise}, 
        \end{dcases}
    \end{equ}
     
	Since $\IE \{D \vert W\} = \lambda (W + R)$,  and $D = F(X, X') $ is antisymmetric,  it follows that, for any absolutely continuous function $f$, 
    \begin{align*}
        0 & =  \IE \bigl\{D \bigl( f(W) + f(W') \bigr)\bigr\} \\
          & = 2 \IE \{D f(W)\} - \E \bigl\{D \bigl( f(W) - f(W') \bigr)\bigr\} \\
          & = 2 \lambda \IE \{(W + R) f(W)\} - \IE \biggl\{D \int_{-\Delta}^0 f'(W + u) \dd u\biggr\} . 
    \end{align*}
    Rearranging the foregoing equality, we have 
    \begin{equ}
        \label{eq:ewfw}
        \E \{W f(W)\} =\frac{1}{2\lambda}  \E \biggl\{D \int_{-\Delta}^{0} f'(W + u) \dd u\biggr\} - \E \{R f(W)\}. 
    \end{equ}    
    By \cref{eq:ewfw}, 
    %it follows that
    \begin{align*}
        \E \bigl\{ W f_z(W) \bigr\} = \frac{1}{2\lambda} \E \biggl\{ D \int_{-\Delta}^0 f_z'(W + t) \dd t \biggr\} - \E \bigl\{ R f_z(W) \bigr\}, 
    \end{align*}
    and thus, 
    \begin{equ}
        \label{eq:j123}
         \P (W > z) - \bigl\{ 1 - \Phi(z) \bigr\}  
        & = \E \bigl\{ f_z'(W) - W f_z(W) \bigr\} \\
        % & = \E \biggl\{ f_z'(W) \left( 1 - \frac{1}{2\lambda} \conep{ D \Delta}{W} \right) \biggr\}\\
        % &{} \qquad - \frac{1}{2\lambda} \E \biggl\{ D \int_{-\Delta}^0 \left\{ f_z'(W + u) - f_z'(W)  \right\} \dd u  \biggr\} + \E \bigl\{ R f_z(W) \bigr\} \\
        & ={}  J_1 - J_2 + J_3, 
    \end{equ}
    where 
    \begin{align*}
        J_1 & = \E \biggl\{ f_z'(W) \left( 1 - \frac{1}{2\lambda} \conep{ D \Delta}{W} \right) \biggr\}, \\
        J_2 & = \frac{1}{2\lambda} \E \biggl\{ D \int_{-\Delta}^0 \left( f_z'(W + u) - f_z'(W)  \right) \dd u  \biggr\} , \\
        J_3 & = \E \bigl\{ R f_z(W) \bigr\} .
    \end{align*}
    We now bound $J_{1}$, $J_{2}$ and $J_{3}$, separately. By \citet[Lemma 2.3]{Ch11N0}, we have 
    \begin{equ}
		\|f_z\| \leq 1, \quad \|f_z'\| \leq 1, \quad  
		\sup_{z \in \mathbb{R}} \bigl\lvert w f(w) \bigr\rvert \leq 1. 
		\label{eq:4.7}
    \end{equ}
    Therefore, 
    \begin{equ}
        \label{eq:j13} 
        |J_1| & \leq \E \biggl\lvert 1 - \frac{1}{2\lambda} \conep{D \Delta}{ W } \biggr\rvert, \\
        |J_3| & \leq \E|R|.
    \end{equ}

    For $J_2,$ observe that $f_z'(w) = w f(w) - \IN{ w > z } + \bigl\{ 1 - \Phi(z) \bigr\}$, and both $wf_z(w)$ and $\IN{ w > z }$ are increasing functions (see, e.g. \citet[Lemma 2.3]{Ch11N0}), by \cref{lem:psi}, 
    \begin{equ}
        \label{eq:j2122} 
        |J_2| &\leq \frac{1}{2\lambda} \biggl\lvert \E \biggl\{ D \int_{-\Delta}^0 \left\{ (W + u)f_z(W + u) - W f_z'(W)  \right\} \dd u  \biggr\}\biggr\rvert \\
              & \quad +  \frac{1}{2\lambda} \biggl\lvert \E \biggl\{ D \int_{-\Delta}^0 \left\{ \IN{ W + u > z } - \IN{ W > z } \right\} \dd u  \biggr\}\biggr\rvert \\
              & \leq \frac{1}{2\lambda} \E \bigl\{\bigl\lvert \conep{D^{*} \Delta}{W} \bigr\rvert \bigl( |W f_z(W)| + \IN{ W > z } \bigr) \bigr\}\\
              & \leq J_{21} + J_{22}, 
    \end{equ}
    where 
    \begin{align*}
        J_{21} & =  \frac{1}{2\lambda} \E \Bigl\{\bigl\lvert \conep{D^{*} \Delta}{W} \bigr\rvert \cdot \bigl\lvert W f_z(W) \bigr\rvert\Bigr\}, \\
        J_{22} & = \frac{1}{2\lambda} \E \Bigl\{\bigl\lvert \conep{D^{*} \Delta}{W} \bigr\rvert \IN{ W > z }\Bigr\} . 
    \end{align*}
    Then, by \cref{eq:4.7}, $|J_2| \leq \frac{1}{\lambda} \E \bigl\lvert \conep{D^{*} \Delta}{W} \bigr\rvert$. This proves 
    \cref{thm01} together with \cref{eq:j13}. 
\end{proof}
\subsection{Proof of \cref{cor-nonlinear}}%
\label{sub:proof_of_cor-nonlinear}
In this subsection, we apply \cref{thm01} to prove \cref{cor-nonlinear}. By \cref{eq-corcon}, we have 
\begin{align*}
	\E \{D \mvert X\} = \lambda(T + U - R). 
\end{align*}
Let $T' = \phi(X') + U(X')$, then we have $(T, T')$ is exchangeable. Then, by \cref{thm01}, we have 
\begin{align*}
	\MoveEqLeft \sup_{z \in \R} \bigl\lvert \P [ T \leq z ] - \Phi(z) \bigr\rvert\\
	& \leq \E \bigl\lvert 1 - \frac{1}{2\lambda} \E \{ D (T - T') \mvert X \} \bigr\rvert \\
	& \quad  + \frac{1}{\lambda} \E \bigl\lvert \E \{ D^* (T - T') \mvert X \} \bigr\rvert + \E |U| + \E |R| \\
	& \leq \E \bigl\lvert 1 - \frac{1}{2\lambda} \E \{ D (\phi(X) - \phi(X')) \mvert X \} \bigr\rvert \\
	& \quad  + \frac{1}{\lambda} \E \bigl\lvert \E \{ D^* (\phi(X) - \phi(X')) \mvert X \} \bigr\rvert + \E |U| + \E |R| + \frac{3}{2\lambda} \E | D (R - R')|. 
\end{align*}
This completes the proof by recalling that $\Delta = \phi(X) - \Phi(X')$.

\section{Proofs of Theorems 2.1, 2.3 and 2.4}%
\label{sec:proofs}
In this section, we give the proofs of \cref{thm-3.1,thm-3.0,thm-3.2}.

\subsection{Proof of \cref{thm-3.1}}%
\label{sub:proof_of_theorem_3_1}
% [Proof of \cref{thm-3.1}]
Without loss of generality, we assume that $n \geq \max(2, k^2)$, otherwise the inequality is trivial. 
We use \cref{cor-nonlinear} to prove this theorem. For each $\alpha = (\alpha(1), \dots, \alpha(k)) \in \mathcal{I}_{n,k}$, let
\begin{equ}
	\MoveEqLeft r(X_{\alpha(1)}, \dots, X_{\alpha(k)}; Y_{\alpha(1), \alpha(2)}, \dots, Y_{\alpha(k - 1), \alpha(k)}) \\
	& = \func(X_{\alpha(1)}, \dots, X_{\alpha(k)}; Y_{\alpha(1), \alpha(2)}, \dots, Y_{\alpha(k - 1), \alpha(k)}) - \sum_{j = 1}^k f_1(X_{\alpha(j)}). 
    \label{t3.1-01}
\end{equ}
Let $\sigma_n = \sqrt{\Var \{ S_{n,k}(f) \}}$, and 
\begin{align*}
	T = \frac{1}{\sigma_n}( S_{n,k}(f) - \E \{ S_{n,k}(f) \} ) = W  + U, 
\end{align*}
where 
\begin{align*}
	W & =  \frac{1}{ \sigma_n } \binom{n-1}{n-k} \sum_{i = 1}^n f_1(X_i), \\
	U & =  \frac{1}{\sigma_n}\sum_{\alpha \in \mathcal{I}_{n,k}} r (X_{\alpha(1)}, \dots, X_{\alpha(k)}; Y_{\alpha(1), \alpha(2)}, \dots, Y_{\alpha(k - 1), \alpha(k)} ).
\end{align*}

By orthogonality we have $\Cov(W, U) = 0$, and thus
\begin{equ}
	\sigma_n^2 \geq \Var(W) =  \binom{n-1}{n-k}^2 \Var \biggl( \sum_{j =1}^n f_1(X_j) \biggr) = \binom{n}{k}^2 \frac{k^2 \sigma_1^2}{n}. 
    \label{t31-sig}
\end{equ}
 
Let $(X_1', \dots, X_n')$ be an independent copy of $(X_1, \dots, X_n)$.
For each $i = 1, \dots, n$, define $X^{(i)} = (X_1^{(i)}, \dots, X_n^{(i)})$ where 
\begin{align*}
	X_j^{(i)} = 
	\begin{cases}
		X_j & \text{if $j \neq i$,}\\
		X_i' & \text{if $j = i$,}
	\end{cases}
\end{align*}
and let 
\begin{align*}
	U^{(i)} = \frac{1}{\sigma_n}\sum_{\alpha \in \mathcal{I}_{n,k}} r ( X_{\alpha(1)}^{(i)}, \dots, X_{\alpha(k)}^{(i)}; Y_{\alpha(1), \alpha(2)}, \dots, Y_{\alpha(k - 1), \alpha(k)} ).
\end{align*}
The following lemma provides the upper bounds of $\E \{ R_1^2 \}$ and $\E \{ (R_1 - R_1^{(i)})^2 \}$.
\begin{lemma}
	\label{lem3.3}
	For $n \geq 2$ and $k \geq 2$,  
	\begin{align}
		\E \{ U^2 \} & \leq \frac{(k - 1)^2 \tau^{2}}{2 (n - 1) \sigma_1^2} \label{l3.3-aa} \\
		\E \{ ( U - U^{(i)} )^2 \} & \leq \frac{2 (k - 1)^2 \tau^{2}}{n(n - 1) \sigma_1^2}. \label{l3.3-bb}
	\end{align}
\end{lemma}

The proof of \cref{lem3.3} is put in the appendix. 

Now, we apply \cref{cor-nonlinear} to prove the Berry--Esseen bound for $T$. To this end, let $\xi_i = \sigma_n^{-1}f_1(X_i)$ for each $1 \leq i \leq n$. 
Let $I$ be a random index uniformly distributed over $\{1, \dots, n\}$, which is independent of all others.
Let 
\begin{align*}
	D = \Delta = \frac{1}{\sigma_n} \binom{n - 1}{n - k} \bigl( f_1(X_I) - f_1(X_I') \bigr), 
\end{align*}
then it follows that 
\begin{align*}
	\E \{ D \mvert W \} = \frac{1}{n} W.  
\end{align*}
Thus, \cref{eq-corcon} is satisfied with $\lambda = 1/n$ and $R = 0$. Moreover, we have  
\begin{align*}
	\frac{1}{2\lambda}\E \{ D\Delta \mvert X \}
	& = \frac{1}{2\sigma_n^2} \binom{n-1}{n-k}^2 \sum_{i = 1}^n \bigl( f_1(X_i) - f_1(X_i') \bigr)^2, \\
	\frac{1}{\lambda}\E \{ |D| \Delta \mvert X \}
	& = \frac{1}{\sigma_n^2} \binom{n-1}{n-k}^2 \sum_{i = 1}^n \bigl( f_1(X_i) - f_1(X_i') \bigr) \bigl\lvert f_1(X_i) - f_1(X_i') \bigr\rvert. 
\end{align*}
Also, 
\begin{align*}
	\frac{1}{2\lambda}\E \{ D \Delta \} & = \E \{W^2\} = 1 - \E \{U^2\}, \quad \E \{|D| \Delta\} = 0. 
\end{align*}

Therefore, by the Cauchy inequality and \cref{lem3.3}, we have for $n \geq \max(2, k^2)$, 
\begin{align*}
	\MoveEqLeft \E \biggl\lvert \frac{1}{2\lambda} \E \{D \Delta \mvert X\} - 1 \biggr\rvert\\
	& \leq \E \biggl\lvert \frac{1}{2\lambda} \E \{D \Delta \mvert X\} - \frac{1}{2\lambda} \E \{D \Delta\} \biggr\rvert + \E \{U^2\} \\
	& \leq \frac{1}{2\sigma_n^2} \binom{n-1}{n-k}^2 \biggl( \Var \biggl\{ \sum_{i=1}^n \bigl( f_1(X_i) - f_1(X_i') \bigr)^2 \biggr\} \biggr)^{1/2} + \frac{(k-1)^2 \tau^2}{2(n-1) \sigma_1^2}\\
	& \leq \frac{2 \tau^2}{\sqrt{n}\sigma_1^2 } + \frac{(k-1) \tau^2}{\sqrt{n} \sigma_1^2}  \leq \frac{(k+1) \tau^2}{\sqrt{n} \sigma_1^2}, 
\end{align*}
where we used \cref{t31-sig} in the last line. 
Using the same argument, we have for $n \geq \max \{2, k^2\}$,  
\begin{align*}
	\MoveEqLeft \E \biggl\lvert \frac{1}{\lambda} \E \{|D| \Delta \mvert X\}  \biggr\rvert
	\\
	& \leq \frac{1}{\sigma_n^2} \binom{n-1}{n-k}^2 \biggl( \Var \biggl\{ \sum_{i=1}^n \bigl( f_1(X_i) - f_1(X_i') \bigr) \biggr\} \biggr)^{1/2}\\
	& \leq \frac{4 \tau^2}{\sqrt{n}\sigma_1^2 }. 
\end{align*}

Now we give the bounds for $U $ and $U^{(i)}$. We have two cases. 
For the case where $k = 1$, then it follows that 
$U = 0$ and $U^{(i)} = 0$. 
As for $k \geq 2$, 
noting that  $(n - 1)^{-1/2} \leq 2 n^{-1/2} $ for $n \geq 2$,
by \cref{lem3.3} and the Cauchy inequality, we have 
\begin{align*}
	\E \lvert U \rvert & \leq \frac{0.71 (k - 1) \tau}{(n - 1)^{1/2} \sigma_1} \leq \frac{2 (k - 1) \tau}{\sqrt{n}  \sigma_1}  , 
\end{align*}
and 
\begin{align*}
	\sum_{i = 1}^n \E \{ \lvert (\xi_i - \xi_i') (U - U^{(i)} )\rvert \}
	& \leq \frac{2.84 (k - 1) \tau}{(n - 1)^{1/2} \sigma_1} \leq \frac{6 (k - 1) \tau}{\sqrt{n}  \sigma_1}.
\end{align*}
By \cref{cor-nonlinear} and noting that $\sigma_1^2 \leq \E \{ \func( X_{\{\alpha\}} ; Y_{\{\alpha\}})^2 \} \leq \tau^{1/2}$, we have 
\begin{align*}
	\sup_{z \in \R} \bigl\lvert \P [ T \leq z ] - \Phi(z) \bigr\rvert 
	& \leq   \frac{(k+5)\tau^{2}}{\sqrt{n} \sigma_1^2 } + \frac{11(k - 1) \tau}{\sqrt{n}  \sigma_1}  \\
	& \leq \frac{12k \tau^{2}}{\sqrt{n} \sigma_1^2 } .
\end{align*}
This proves \cref{t3.1-ab}. 

% By \citet{janson1991asymptotic}, any symmetric function $f : \mathcal{X}^k \times \mathcal{Y}^{k(k -1)/2}$ has an orthogonal series expansion.  We first introduce some notation. For any $m \in \mathcal{I}_{n,k}$, $G \subset K_m$ (recall that $K_m$ is the complete graph
% generated by $\alpha_1, \dots, m_k$), let $\mathcal{F}_G$ be the $\sigma$-field generated by $\tilde{X}_G \coloneqq \{ X_i \}_{i \in V(G)}$ and $\tilde{Y}_G \coloneqq \{ Y_{i,j} \}_{(i,j) \in B}$. Let $\mathcal{L}_G^2$ be the Hilbert space of all square integrable random
% variables that are functions of $X_i$ and $Y_{i,j}$ such that $i \in V(G)$ and $(i,j) \in B$. Now, introduce the subspaces of $\mathcal{L}_G^2$: 
% \begin{align*}
% 	\mathcal{M}_G = \{ f \in \mathcal{L}_G^2 : \E \{ h f  \} = 0 \text{ for every $h$ such that $g \in \mathcal{L}_H^2$ and $H \subsetneqq G$} \}.
% \end{align*}
% The following lemma says that the space $\mathcal{L}_G^2$ is the orthogonal direct sum of $\{ \mathcal{M}_H : H \subset G\}$.
% \begin{lemma}[\citet{janson1991asymptotic},~Lemma~1]
% 	For any $f \in \mathcal{L}_G^2$, we have 
% 	\begin{align*}
% 		f_G = \sum_{H \subset G} f_H. 
% 	\end{align*}
% \end{lemma}

\subsection{Proof of Theorem 2.2}%
\label{sub:p3.2}

We first prove a proposition for the Hoeffding decomposition. 
\begin{proposition}
	\label{lem-psiG}
	For $A \subset [n], B \subset [n]_2$ such that $(A, B) \neq (\emptyset, \emptyset)$, and for any $\tilde{A}, \tilde{B}$ such that $\tilde{A} \subset A$ and $\tilde{B} \subset B$ but $(\tilde{A}, \tilde{B}) \neq (A,B)$, we have 
	\begin{equ}
		\IE \bigl\{ f_{A, B}(X_{A}; Y_{B}) \bigm\vert X_{\tilde{A}}, Y_{\tilde{B}}\bigr\} = 0. 
        \label{eq-pp1}
	\end{equ}
\end{proposition}
\begin{proof}
If $|A| + |B| = 1$, then for $(\tilde{A}, \tilde{B}) = (\emptyset, \emptyset)$, by definition, 
\begin{align*}
	\IE \bigl\{ f_{A, B}(X_{A}; Y_{B}) \bigm\vert X_{\tilde{A}}, Y_{\tilde{B}}\bigr\}
	& = \E \bigl\{ f_{A, B}(X_{A}; Y_{B})\bigr\}\\
	& = \E \{ f(X_{[k]}; Y_{[k]_2}) \} - \E \{ f(X_{[k]}; Y_{[k]_2}) \} = 0.
\end{align*}
We prove the proposition by induction. Assume that \cref{eq-pp1} holds for $1 \leq |A| + |B| \leq  m$. 
Let $\mathcal{A}_{\tilde{A}, \tilde{B}} = \{ (A', B') : A' \subset \tilde{A} , B' \subset \tilde{B} , \}$ and let $\mathcal{A}_{\tilde{A},\tilde{B}}^c = \{ (A', B') : A' \subset A, B' \subset B, (A', B') \neq (A, B) \} \setminus \mathcal{A}_{\tilde{A}, \tilde{B}}^c$. 
Reordering \cref{eq-fAB} by the inclusive-exclusive formula we have 
\begin{align*}
	f_{A,B}(X_{A}; Y_{B})
	& =  \E \{ f(X_{[k]}; Y_{[k]_2}) \vert X_{A}, Y_{B} \}
	- \sum_{ |A'| + |B'| < |A| + |B| } f_{A', B'} (X_{A'} ; Y_{B'})\\
	& = \E \{ f(X_{[k]}; Y_{[k]_2}) \vert X_{A}, Y_{B} \}
	- \sum_{(A', B') \in  \mathcal{A}_{\tilde{A},\tilde{B}} } f_{A', B'} (X_{A'} ; Y_{B'}) \\
	& \quad - \sum_{(A', B') \in  \mathcal{A}_{\tilde{A},\tilde{B}}^c } f_{A', B'} (X_{A'} ; Y_{B'}) \\
	& = \E \{ f(X_{[k]}; Y_{[k]_2}) \vert X_{A}, Y_{B} \} - \E \{ f(X_{[k]}; Y_{[k]_2}) \vert X_{\tilde{A}}, Y_{\tilde{B}} \} \\
	& \quad - \sum_{(A', B') \in  \mathcal{A}_{\tilde{A},\tilde{B}}^c } f_{A', B'} (X_{A'} ; Y_{B'}).
\end{align*}
By the induction assumption, we have 
\begin{align*}
	\sum_{(A', B') \in  \mathcal{A}_{\tilde{A},\tilde{B}}^c } \E \{ f_{A', B'} (X_{A'} ; Y_{B'}) \vert X_{\tilde{A}}, Y_{\tilde{B}} \} = 0. 
\end{align*}
Then, the desired result follows. 
\end{proof}

Let 
\begin{align*}
	\mathcal{A}_{n,\ell} = \{ \alpha = (\alpha(1), \dots, \alpha(\ell)) : 1 \leq \alpha(1) \neq \dots \neq \alpha(\ell) \leq n \}. 
\end{align*}
Then, $\mathcal{I}_{n,\ell} \subset \mathcal{A}_{n, \ell}$.
For $A \subset [\ell]$ and $B \in [\ell]_2$ and $\alpha = (\alpha(1), \dots, \alpha(\ell)) \in \mathcal{A}_{n,\ell}$, write 
\begin{align*}
	\alpha (A) & =  ( \alpha(i) )_{ i \in A }, & \alpha (B) & =  \bigl( ( \alpha(i), \alpha(j))  \bigr)_{(i,j) \in B}, \\
	X_{\alpha (A)} & =  ( X_{i} )_{ i \in \alpha(A) }, & Y_{\alpha (B)} & =  ( Y_{ i, j} )_{ (i,j) \in \alpha(B)}. 
\end{align*}
Moreover, for any $\alpha \in \mathcal{I}_{n,\ell}$ and $f_{A,B} : \mathcal{X}^{|A|} \times \mathcal{Y}^{|B|} \to \R$, let  
\begin{align*}
	\tilde{S}_{n,\ell} (f_{A,B}) = \sum_{\alpha \in \mathcal{A}_{n,\ell}} f_{A,B}( X_{\alpha(A)}; Y_{\alpha(B)} ), 
\end{align*}
and similarly, $S_{n,\ell}(f_{A,B})$ can be represented as   
\begin{math}
	\sum_{\alpha \in \mathcal{I}_{n,\ell}} f_{A,B} ( X_{\alpha(A)}; Y_{\alpha(B)} ).  
\end{math}

Let $(Y_{1,1}', \dots, Y_{n-1, n}')$ be an independent copy of $Y = (Y_{1,1}, \dots, Y_{n-1,n})$. 
For any $(i,j) \in \mathcal{A}_{n,2}$, let $Y^{(i,j)} = (Y_{1,1}^{(i,j)}, \dots, Y_{n-1, n}^{(i,j)})$ with
\begin{align*}
	Y_{p, q}^{(i,j)} = 
	\begin{cases}
		Y_{p, q}  & \text{ if $\{p, q\} \neq \{i, j\}$,}\\
		Y_{p, q}' & \text{ if $\{p, q\} = \{i, j\}$,}
	\end{cases}
	\quad 
	\text{for }
	(p, q) \in \mathcal{I}_{n,2}.
\end{align*}
Then, it follows that for each $(i,j) \in \mathcal{A}_{n,2}$, $((X, Y), (X, Y^{(i,j)}))$ is an exchangeable pair. 
For any $B \subset [n]_2$, let  
$Y_B^{(i,j)} = (Y_{p,q}^{(i,j)})_{ (p,q) \in B}$. 
For any $A \subset [\ell]$, $B \subset [\ell]_2$, $\alpha = ( \alpha(1), \dots, \alpha(\ell) )\in \mathcal{I}_{n,\ell}$ and  $f_{A,B} : \mathcal{X}^{ |A| } \times \mathcal{Y}^{|B|} \to \R$, define 
\begin{align*}
	Y_{\alpha (B)}^{(i,j)} & =  ( Y_{\alpha(p), \alpha(q)}^{(i,j)} )_{ (p,q) \in B }, \\
	S_{n,\ell}^{(i,j)} (f_{A,B}) & =  \sum_{\alpha \in \mathcal{I}_{n,\ell}} f_{A,B} (X_{\alpha (A)}; Y_{\alpha(B)}^{(i,j)}), \\   
	\tilde{S}_{n,\ell}^{(i,j)} (f_{A,B}) & =  \sum_{\alpha \in \mathcal{A}_{n,\ell}} f_{A,B}( X_{\alpha(A)}, Y_{\alpha(B	)}^{(i,j)}) . 
\end{align*}

% Then, it follows that $S_{n,\ell}(f) = \tilde{S}_{n,\ell}(f)/\ell!$ for any symmetric function $f$.

	Let $\func_{(\ell)}$ be defined as in \cref{eq-psiell}, and it follows that 
	\begin{align*}
		\func = \sum_{\ell = 0}^{k} \func_{(\ell)}, \quad \func_{(0)} = \E \{ \func(X_{[k]}; Y_{[k]_2}) \}, \quad S_{n,k} (\func_{(0)}) = \E \{ S_{n,k}(\func) \}.
	\end{align*}	
	Moreover, by
	assumption, as $\func$ has principal degree $d$, and it follows that $\func_{(\ell)} \equiv 0$ for $\ell = 1, \dots, d - 1$.  
	Let $\sigma_{n} = ( \Var \{ S_{n,k}(\func) \} )^{1/2}$ and $\sigma_{n,\ell} = (\Var \{ S_{n,k} (\func_{(\ell)}) \})^{1/2}$. 
	The next lemma estimates the upper and lower bounds of $\sigma_n^2$ and $\sigma_{n,d}^2$. 
The proof is similar to that of Lemma 4 of \citet{janson1991asymptotic}, and we omit the details. 
\begin{lemma}
	\label{lem-3.w}
	We have for each $(i,j) \in \mathcal{A}_{n,2}$ and $d \leq \ell \leq k$, 
	\begin{align}
		& \sigma_{n,\ell}^2  = \sum_{(A, B): v_{A,B} = \ell} \frac{n!(n-\ell)!\sigma_{A,B}^2}{ (n-k)!^2 (k-\ell)!^2\lvert \Aut(G_{A,B})\rvert } \leq C n^{2k-\ell}\tau^2 , \label{l3.w-aa} \\
		& \sigma_{n}^2  = \sum_{\ell = d}^{k}\sum_{(A,B): v_{A,B} = \ell} \frac{n!(n-\ell)!\sigma_{A,B}^2}{ (n-k)!^2 (k-\ell)!^2\lvert \Aut(G_{A,B})\rvert }  \leq C n^{2k-d}\tau^2 , \label{l3.w-bb}\\
		& \E \{ (S_{n,k}(\func_{(\ell)}) - S_{n,k}^{(i,j)}(\func_{(\ell)}))^2 \}  \leq C n^{2k - \ell - 2} \tau^2, \label{l3.w-cc}
	\end{align}
	and 
	\begin{equ}
		\sigma_{n}^2 \geq \sigma_{n,d}^2 \geq c n^{2k - d} \sigma_{\min}^2.
        \label{l3.w-dd}
	\end{equ}
	where $\lvert \Aut(G) \rvert$ is the number of the automorphisms of $G$, and  $c, C>0$ are some absolute constant.
\end{lemma}

For any $A \subset [k]$ and $B \subset [k]_2$, let
\begin{align*}
	\mu_{A,B} &\coloneqq \frac{1}{ \lvert \Aut(G_{A,B}) \rvert |B| }\binom{n - v_{A,B}}{n - k} , \\ 
	\nu_{A,B} &\coloneqq |B| \times \mu_{A,B} =  \frac{1}{ \lvert \Aut(G_{A,B}) \rvert }\binom{n - v_{A,B}}{n - k} ,  
\end{align*}
and for any $\alpha \in \mathcal{A}_{n,\ell}$ ($\ell = 1, \dots, k$), let 
\begin{align*}
	\xi_{\alpha(A,B)}^{(i,j)}	= f_{A,B}( X_{\alpha(A)}; Y_{\alpha(B)} ) - f_{A,B}(X_{\alpha(A)}; Y_{\alpha(B)}^{(i,j)}).
\end{align*}

Recall that $G_{A,B}$ is the graph generated by $(A,B)$. 
For any $(A_j, B_j)$ for $j  = 1, 2$, we simply write $v_j = v_{A_j, B_j}$ as the number of nodes of the graph $G_{A_j, B_j}$. Recall that $\mathcal{G}_{f,d} = \{ (A, B) : A \subset [d] , B \subset [d]_2, \sigma_{A,B} > 0, v_{A,B} = d \}$ and we similarly define $\mathcal{G}_{f,d +
	1} = \{
(A, B) : A \subset [d + 1] , B \subset [d + 1]_2, \sigma_{A,B} > 0, v_{A,B} = d + 1 \}$.
We have the following lemmas. 
\begin{lemma}
	\label{lem5.3}
	For all $(A_1, B_1)$, $(A_2, B_2) \in \mathcal{G}_{f,d}$  such that $G_{A_1, B_1}$ and $G_{A_2, B_2}$ are connected, we have 
	\begin{align*}
		\Var \Biggl\{ \sum_{(i,j) \in \mathcal{A}_{n,2}} \biggl(\sum_{\alpha_1 \in \mathcal{A}_{n,d}^{(i,j)}} \xi_{\alpha_1(A_1, B_1)}^{(i,j)}\biggr) \biggl(\sum_{\alpha_2 \in \mathcal{A}_{n,d}^{(i,j)}} \xi_{\alpha_2(A_2, B_2)}^{(i,j)}\biggr) \Biggr\} \leq C n^{2d - 1}
	\tau^{4}. 
	\end{align*}
\end{lemma}

\begin{lemma}
	\label{lem5.5}
	Assume that $k \geq d + 1$.
	For all $(A_1, B_1)$, $(A_2, B_2) \in \mathcal{G}_{f, d}\cup \mathcal{G}_{f, d + 1}$, we have 
	\begin{align*}
		\Var \Biggl\{ \sum_{(i,j) \in \mathcal{A}_{n,2}} \biggl( \sum_{\alpha_1 \in \mathcal{A}_{n,v_1}^{(i,j)}} \xi_{\alpha_1(A_1,B_1)}^{(i,j)} \biggr) \biggl\lvert \sum_{\alpha_2 \in \mathcal{A}_{n,v_2}^{(i,j)}} \xi_{\alpha_2(A_2,B_2)}^{(i,j)} \biggr\rvert \Biggr\}
		\leq C n^{2 \max \{v_1,v_2\} - 2} \tau^4. 
	\end{align*}
\end{lemma}

We are now ready to give the proof of \cref{thm-3.0}.
\begin{proof}
[Proof of \cref{thm-3.0}]
We assume that $n \geq \max \{k,2\}$ without loss of generality, otherwise the result is trivial.  
Recall that $f_{(d)}$ is defined in \cref{eq-psiell}.
Write $T = \sigma_{n}^{-1} (S_{n,k}(\func) - \E \{S_{n,k}(\func)\})$, and 
\begin{equ}
	W = \sigma_{n}^{-1} S_{n,k} (\func_{(d)}), \quad U = T - W = \sigma_{n}^{-1} \sum_{\ell = d + 1}^k S_{n,k}(\func_{(\ell)}). 
	\label{t3.0-04}
\end{equ}
Here, if $d + 1 > k$, then set $\sum_{\ell = d + 1}^k S_{n,k}(\func_{(\ell)}) = 0$.
With a slight abuse of notation, we write $(A,B) \in \mathcal{G}_{f,d}$ if $G_{A,B} \in \mathcal{G}_{f,d}$.
We have 
\begin{align*}
	W & =  \frac{1}{\sigma_n} \sum_{\alpha \in \mathcal{A}_{n,d}} \sum_{(A,B) \in \mathcal{G}_{f,d}}  \binom{n - d}{k - d} \frac{ \func_{A,B}(X_{\alpha(A)}; Y_{\alpha(B)})  }{ \lvert \Aut(G_{A,B}) \rvert }\\
	& = 
	\frac{1}{\sigma_n} \sum_{\alpha \in \mathcal{A}_{n,d}} \sum_{({A,B}) \in \mathcal{G}_{f,d}}  \binom{n - d}{k - d} \frac{ \func_{A,B	}(X_{\alpha(A)}; Y_{\alpha(B)})  }{ \lvert \Aut(G_{A,B}) \rvert }, 
\end{align*}
because by assumption, $\func_{A,B} \equiv 0$ for all $(A,B) \in \mathcal{G}_{f,d}$. 

For each $(i,j) \in \mathcal{A}_{n,2}$, let 
\begin{align*}
	W^{(i,j)} = \frac{1}{\sigma_{n}} S_{n,k}^{(i,j)} (\func_{(d)}), \quad 
	U^{(i,j)} = \sigma_{n}^{-1} \sum_{\ell = d + 1}^k S_{n,k}^{(i,j)} (\func_{(\ell)}). 
\end{align*}

Let $(I,J)$ be a random 2-fold index
uniformly chosen in $\mathcal{A}_{n,2}$, which is independent of all others. 
Then, $( (X, Y), (X, Y^{(I,J)}))$ is an exchangeable pair.
Let 
\begin{align*}
	\Delta & =  W - W^{(I,J)} =  \frac{1}{\sigma_{n}} \sum_{\alpha \in \mathcal{A}_{n,d}} \sum_{(A,B) \in \mathcal{G}_{f,d}}  \nu_{A,B} \xi_{\alpha(A,B	)}^{(I,J)}.
\end{align*}
Also, define 
\begin{align*}
	D = \frac{1}{\sigma_{n}} \sum_{\alpha \in \mathcal{A}_{n,d}} \sum_{(A,B) \in \mathcal{G}_{f,d}	}  \mu_{A,B} \xi_{\alpha(A,B)}^{(I,J)}.
\end{align*}
Then, we have $D$ is antisymmetric with respect to $(X, Y)$ and $(X, Y^{(I,J)})$.

Let $\mathcal{A}_{n,d}^{(i,j)} = \{ \alpha \in \mathcal{A}_{n,d} : \{i,j\} \subset \{\alpha\} \}$. Then, 
\begin{align*}
	\MoveEqLeft\E \{ D \mvert X, Y \}\\
	& = \frac{1}{n(n-1)\sigma_{n}} \sum_{(i,j) \in \mathcal{A}_{n,2}} \sum_{\alpha \in \mathcal{A}_{n,d}^{(i,j)}} \sum_{(A,B) \in \mathcal{G}_{f,d}} \mu_{A,B} \E \bigl\{ \xi_{\alpha(A,B)}^{(i,j)} \mvert X, Y\bigr\}.
\end{align*}
By \cref{eq-pp1},  
\begin{align*}
	\MoveEqLeft \E \{ \func_{A,B}(X_{\alpha(A)}; Y_{\alpha(B)}^{(i,j)}) \mvert X, Y \}\\
	& = 
	\E \{ \func_{A,B}(X_{\alpha(A)}; Y_{\alpha(B)}) \mvert X_{A}, Y_{B} \setminus\{ Y_{i, j} \} \} \\
	& =  
	\begin{cases}
		0 & \text{if $(i,j) \in B$, }\\
		\func_{A,B}(X_{\alpha(A)}; Y_{\alpha(B)}) & \text{otherwise}.
	\end{cases}
\end{align*}
Moreover, note that for $\alpha \in \mathcal{A}_{n,d}$,
\begin{align*}
	\sum_{(i,j) \in \mathcal{A}_{n,2}}  \IN{ (i,j) \in \alpha(B) } = 2 |B|, 
\end{align*}
and thus
\begin{align}
	\MoveEqLeft \E \{ D \mvert X, Y \}\nonumber  \\
	& =  \frac{1}{n(n-1) \sigma_{n}} \sum_{\alpha \in \mathcal{A}_{n,d}} \sum_{(A,B)\in \mathcal{G}_{f,d}} \mu_{A,B} \func_{A,B}(X_{\alpha(A)}; Y_{\alpha(B)}) \\
	& \quad \times \sum_{ (i, j) \in \mathcal{A}_{n,2} }  \IN { (i,j) \in  \alpha(B) } \nonumber \\
	& = \frac{2}{n(n-1)\sigma_{n}}\sum_{\alpha \in \mathcal{A}_{n,d}} \sum_{(A,B)\in \mathcal{G}_{f,d}} \nu_{A,B} \func_{A,B}(X_{\alpha(A)}; Y_{\alpha(B)})\nonumber \\
	& = \frac{2}{n(n-1)} W. 
    \label{t3.0-03}
\end{align}
 
Thus, \cref{eq-corcon} is satisfied with $\lambda = 2/(n (n - 1))$ and $R = 0$.
Moreover, by exchangeability, 
\begin{align} \label{t3.0-02}
	\E \{ D \Delta \} = 2 \E \{ D W \} = 2 \lambda \E \{W^2\} = 2 \lambda \sigma_{n,d}^2/\sigma_n^2 .
\end{align}
Then, we have   
\begin{align*}
	\MoveEqLeft\frac{1}{2\lambda} \E \{ D \Delta \mvert X, Y, Y' \}\\
	& = \frac{1}{4\sigma_{n}^{2}}\sum_{(A_1,B_1)\in \mathcal{G}_{f,d}} \sum_{(A_2,B_2)\in \mathcal{G}_{f,d}} \mu_{A_1,B_1} \nu_{A_2,B_2} \\
	& \quad \times \sum_{(i,j) \in \mathcal{A}_{n,2}}  
	\biggl(\sum_{\alpha \in \mathcal{A}_{n,d}^{(i,j)}}  \xi_{\alpha(A_1,B_1)}^{(i,j)}  \biggr) \biggl(\sum_{\alpha \in \mathcal{A}_{n,d}^{(i,j)}}
	\xi_{\alpha(A_2,B_2)}^{(i,j)}  \biggr) .
\end{align*}
Now, by the Cauchy inequality, \cref{t3.0-02} and \cref{lem-3.w,lem5.3}, we have 
\begin{align*}
	\MoveEqLeft \E \biggl\lvert \frac{1}{2\lambda} \E \{ D \Delta \mvert X, Y, Y' \} - 1 \biggr\rvert\\
	& \leq \E \biggl\lvert \frac{1}{2\lambda} \E \{ D \Delta \mvert X, Y, Y' \} - \frac{1}{2\lambda}\E \{D \Delta\}\biggr\rvert + \frac{ \sigma_n^2 - \sigma_{n,d}^2 }{\sigma_{n}^2}\\
	& \leq \frac{1}{4\sigma_{n}^2} \sum_{(A_1,B_1)\in \mathcal{G}_{f,d}} \sum_{(A_2,B_2)\in \mathcal{G}_{f,d}} \mu_{A_1,B_1} \nu_{A_2,B_2} \\
	& \quad \times \Biggl( \Var \Biggl\{ \sum_{(i,j) \in \mathcal{A}_{n,2}} \biggl( \sum_{\alpha \in \mathcal{A}_{n,d}^{(i,j)}} \xi_{\alpha(A_1,B_1)}
				^{(i,j)} \biggr) \biggl( \sum_{\alpha \in
	\mathcal{A}_{n}^{(i,j)}} \xi_{\alpha(A_2,B_2)}^{(i,j)} \biggr) \Biggr\} \Biggr)^{1/2}\\
	& \quad + \frac{ \sigma_n^2 - \sigma_{n,d}^2 }{\sigma_{n}^2}\\
	& \leq C n^{-1/2}. 
\end{align*}
Taking $D^* = |D|$, by \cref{lem5.5},
\begin{align*}
	\MoveEqLeft \frac{1}{\lambda} \E \bigl\lvert \E \{D^*\Delta  \mvert X , Y, Y'\} \bigr\rvert \\
	& = \frac{1}{\lambda} \E \bigl\lvert \E \{D^*\Delta \mvert X , Y, Y'\} \bigr\rvert  \\
	& \leq  \frac{1}{4\sigma_{n}^2} \sum_{(A_1,B_1)\in \mathcal{G}_{f,d}} \sum_{(A_2,B_2)\in \mathcal{G}_{f,d}} \mu_{A_1,B_1} \nu_{A_2,B_2} \\
	& \quad \times \Biggl( \Var \Biggl\{ \sum_{(i,j) \in \mathcal{A}_{n,2}} \biggl\lvert \sum_{\alpha \in \mathcal{A}_{n,d}^{(i,j)}} \xi_{\alpha(A_1,B_1)}
				^{(i,j)} \biggr\rvert \biggl( \sum_{\alpha \in
	\mathcal{A}_{n}^{(i,j)}} \xi_{\alpha(A_2,B_2)}^{(i,j)} \biggr) \Biggr\} \Biggr)^{1/2}\\
	& \leq C n^{-1}. 
\end{align*}

Now, by \cref{t3.0-04} and \cref{lem-3.w}, we have 
\begin{align*}
	\E|U|^2 & \leq C \sigma_{n,d}^{-2} (k - d) \sum_{\ell = d + 1}^{k} \E ( S_{n,k}^2(\func_{(\ell)})) \leq C n^{-1}, \\ 
	\E (U - U^{(i,j)})^2 & \leq C \sigma_{n,d}^{-2} (k - d) \sum_{\ell = d + 1}^k \E \{ ( S_{n,k}(\func_{(\ell)}) - S_{n,k}^{(i,j)}(\func_{(\ell)}))^2\}
	\leq C n^{-3}, \\
	\E (W - W^{(i,j)})^2 & \leq C \sigma_{n,d}^{-2}  \E \{ ( S_{n,k}(\func_{(d)}) - S_{n,k}^{(i,j)}(\func_{(d)}))^2\}\leq C n^{-2}. 
\end{align*}
Thus, 
\begin{align*}
	\E |U| & \leq C n^{-1/2}, \\
	\frac{1}{\lambda}\E \bigl\lvert \Delta (U - U^{(I,J)})\bigr\rvert
	& = \sum_{i \in \mathcal{I}_{n,2}} \E \{ \lvert (W - W^{(i,j)}) (U - U^{(i,j)})\rvert\}
	\\
	& \leq C n^{-1/2}. 
\end{align*}
Applying \cref{cor-nonlinear}, we obtain the desired result. 
\end{proof}

\subsection{Proof of \cref{thm-3.2}}%
\label{sec:proof_of_thm-3.2}
The proof of \cref{thm-3.2} is similar to that of \cref{thm-3.0}. Without loss of generality, we assume that $k \geq d + 1$, otherwise the proof is even simpler.

For any $A \subset [k]$ and $B \subset [k]_2$, recall that
\begin{align*}
	\mu_{A,B} & \coloneqq \frac{1}{ \lvert \Aut(G_{A,B}) \rvert |B| }\binom{n - v_{A,B}}{n - k} , \\
	\nu_{A,B} & \coloneqq |B| \mu_{A,B} =  \frac{1}{ \lvert \Aut(G_{A,B}) \rvert }\binom{n - v_{A,B}}{n - k} .  
\end{align*}
By \cref{lem-psiG}, we have there exists a Hoeffding decomposition of $g$ as follows:  
\begin{align*}
	g (y)= \sum_{B \subset [k]_2} g_{B}(y_B),  
\end{align*}
where $y = (y_{1,2}, \dots, y_{k - 1, k})$ and $y_{B} = (y_{i,j} : (i,j) \in B).$
Also, for any $B\subset [k]_2$ and $\alpha \in \mathcal{A}_{n,\ell}$ ($\ell = 1, \dots, k$), let 
\begin{align*}
	\eta_{\alpha(B)}^{(i,j)}	= g_{B}(  Y_{\alpha(B)} ) - g_{B}(Y_{\alpha(B)}^{(i,j)}).
\end{align*}
For any $B\in [k]_2$, let $V_B$ be the node set of the graph with edge set $B$. For any $r \in V_B$, let $B^{(r)} = \{ (i, j) : (i,j) \in B, i \neq r, j \neq r \}$. 
Recall that $\mathcal{G}_{f,d+1} = \{ (A,B) : A \subset [k], B\subset [k]_2	, v_{A,B} = d + 1, \sigma_{A,B} > 0 \}$ and $\tilde{\mathcal{G}}_{f,d} = \{ (A, B) \in \mathcal{G}_{f,d} : G_{A,B} \text{ is strongly connected.} \}$.

We need to apply the following lemma in the proof of \cref{thm-3.2}. 
\begin{lemma}
	\label{lem5.4}
	Assume that $k \geq d + 1$. 
	For all $(A_j, B_j) \in \tilde{\mathcal{G}}_{f,d} \cup \mathcal{G}_{f,d + 1}$ and let $v_j = v_{A_j, B_j}$ for $j = 1, 2$, we have 
	\begin{align*}
		\Var \biggl\{ \sum_{(i,j) \in \mathcal{A}_{n,2}} \biggl(\sum_{\alpha_1 \in \mathcal{A}_{n,v_1}^{(i,j)}} \eta_{\alpha_1(B_1)}^{(i,j)}\biggr) \biggl(\sum_{\alpha_2 \in \mathcal{A}_{n,v_2}^{(i,j)}} \eta_{\alpha_2(B_2)}^{(i,j)}\biggr)\biggr\} \leq C
		n^{2d - 2}
	\tau^{4}. 
	\end{align*}
\end{lemma}

\begin{proof}
[Proof of \cref{thm-3.2}] 
Again, write $T = \sigma_{n}^{-1} (S_{n,k}(\fung) - \E \{S_{n,k}(\fung)\})$, and let 
\begin{equ}
	W = \sigma_{n}^{-1} (S_{n,k} (\fung_{(d)}) + S_{n,k}(\fung_{(d + 1)}) ), \quad U = \sigma_{n}^{-1} \sum_{\ell = d + 2}^k S_{n,k}(\fung_{(\ell)}). 
	\label{t3.2-04}
\end{equ}
Here, if $d + 1 > k$, then set $\sum_{\ell = d + 1}^k S_{n,k}(\fung_{(\ell)}) = 0$. Then, $T = W + U$. Now we apply \cref{cor-nonlinear} again to prove the desired result. To this end, we need to construct an exchangeable pair. 
For each $(i,j) \in \mathcal{A}_{n,2}$, let 
\begin{align*}
	W^{(i,j)} = \frac{1}{\sigma_{n}} (S_{n,k}^{(i,j)} (\fung_{(d)}) + S_{n,k}^{(i,j)} (\fung_{(d + 1)}) ), \quad 
	U^{(i,j)} = \sigma_{n}^{-1} \sum_{\ell = d + 2}^k S_{n,k}^{(i,j)} (\fung_{(\ell)}). 
\end{align*}

By assumption, we have 
\begin{align*}
	W & =  \frac{1}{\sigma_n} \sum_{(A,B) \in \tilde{\mathcal{G}}_{f,d} \cup \mathcal{G}_{f,d+1}} \sum_{\alpha \in \mathcal{A}_{n,v(G)}}  \nu_{A,B} { \fung_{A,B}(X_{\alpha(A)}; Y_{\alpha(B)})  }.
\end{align*}

Let $(I,J)$ be a random 2-fold index
uniformly chosen in $\mathcal{A}_{n,2}$, which is independent of all others. 
Then, $( (X, Y), (X, Y^{(I,J)}))$ is an exchangeable pair.
Let 
\begin{align*}
	\Delta & =  W - W^{(I,J)}  =  
		   \frac{1}{\sigma_{n}}  \biggl(\sum_{(A,B) \in \tilde{\mathcal{G}}_{f,d}} \sum_{\alpha \in \mathcal{A}_{n,d}} \nu_{A,B} { \eta_{\alpha(B)}^{(I,J)} } \biggr). 
\end{align*}

Also, define 
\begin{align*}
	D & =  \frac{1}{\sigma_{n}}  \biggl(\sum_{(A,B) \in \tilde{\mathcal{G}}_{f,d}} \sum_{\alpha \in \mathcal{A}_{n,d}} \mu_{A,B} { \eta_{\alpha(B)}^{(I,J)} } \biggr).
\end{align*}
Then, $D$ is antisymmetric with respect to $(X, Y)$ and $(X, Y^{(I,J)})$.

Following a similar argument leading to \cref{t3.0-03}, 
\begin{align}
	\E \{ D \mvert X, Y \} = \frac{2}{n(n-1)} W. 
    \label{t3.2-03}
\end{align}
Thus, \cref{eq-corcon} is satisfied with $\lambda = 2/(n (n - 1))$ and $R = 0$.
Moreover, by exchangeability, 
\begin{align} \label{t3.2-02}
	\E \{ D \Delta \} = 2 \E \{ D W \} = 2 \lambda \E \{W^2\} = 2 \lambda (\sigma_{n,d}^2 + \sigma_{n,d+1}^2)/\sigma_n^2 .
\end{align}
Now, by the Cauchy inequality, \cref{t3.2-02} and \cref{lem-3.w,lem5.4}, we have 
\begin{align*}
	\MoveEqLeft \E \biggl\lvert \frac{1}{2\lambda} \E \{ D \Delta \mvert X, Y, Y' \} - 1 \biggr\rvert\\
	& \leq \E \biggl\lvert \frac{1}{2\lambda} \E \{ D \Delta \mvert X, Y, Y' \} - \frac{1}{2\lambda}\E \{D \Delta\}\biggr\rvert + \frac{ \sigma_n^2 - \sigma_{n,d}^2 - \sigma_{n,d + 1}^2 }{\sigma_{n}^2}\\
	& \leq C n^{-1}. 
\end{align*}
With $D^* = |D|$, and by \cref{lem5.5} again, 
\begin{align*}
	\frac{1}{\lambda} \E \bigl\lvert \E \{D^*\Delta  \mvert X , Y, Y'\} \bigr\rvert \leq C n^{-1}. 
\end{align*}

Now, by \cref{t3.2-04} and \cref{lem-3.w}, we have 
\begin{align*}
	\E|U|^2 & \leq C \sigma_{n}^{-2} (k - d) \sum_{\ell = d + 2}^{k} \E ( S_{n,k}^2(\fung_{(\ell)})) \leq C n^{-2}, \\ 
	\E (U - U^{(i,j)})^2 & \leq C \sigma_{n}^{-2} (k - d) \sum_{\ell = d + 2}^k \E \{ ( S_{n,k}(\fung_{(\ell)}) - S_{n,k}^{(i,j)}(\fung_{(\ell)}))^2\}
	\leq C n^{-4}, \\
	\E (W - W^{(i,j)})^2 & \leq C \sigma_{n}^{-2}  \Bigl\{ \| S_{n,k}(\fung_{(d)}) - S_{n,k}^{(i,j)}(\fung_{(d)}) \|_2^2 \\
						 & \hspace{2.5cm}+ \| S_{n,k}(\fung_{(d + 1)}) - S_{n,k}^{(i,j)}(\fung_{(d + 1)})\|_2^2\Bigr\}\leq C n^{-2}. 
\end{align*}
Thus, 
\begin{align*}
	\E |U| & \leq C n^{-1}, \\
	\frac{1}{\lambda}\E \bigl\lvert \Delta (U - U^{(I,J)})\bigr\rvert
		   & \leq C \sum_{(i,j) \in \mathcal{A}_{n,2}} \E \{ \lvert (W - W^{(i,j)}) (U - U^{(i,j)})\rvert\}
	\\
	& \leq C n^{-1}. 
\end{align*}
Applying \cref{cor-nonlinear}, we obtain the desired result. 
\end{proof}

\section{Proof of other results}%
\label{sec:6}
\subsection{Proof of Theorem 3.2}%
\label{sub:6.1}

As $f_F^{\inj}$ does not dependent on $X$ if $\kappa \equiv p$ for some $0 < p < 1$. 
Fix $F$. Define 
\begin{align*}
	g^{\inj} (Y) = f_F^{\inj}(X; Y)
\end{align*}
and by \cref{lem-psiG}, we have $g^{\inj}$ has the following decomposition: 
\begin{equ}
	g^{\inj}(Y) = \sum_{B \subset [k]_2} g_B^{\inj}(Y_B).
    \label{eq-yinj}
\end{equ}
By \cite[p. 361]{janson1991asymptotic}, we have 
\begin{align*}
	g^{\inj}_{ \{(1,2)\} }(y_{1,2}) = \frac{2 e(F) (v(F) - 2)!}{ \lvert \Aut(G) \rvert } p^{e(F) - 1}(y_{1,2} - p) \neq 0. 
\end{align*}
Therefore, by \cref{thm-3.2} with $d = 2$, we complete the proof.

\subsection{Proof of Theorem 3.3}%
\label{sub:6.2}

Again, let 
\begin{align*}
	g^{\ind}(Y) = f^{\ind}_F(X; Y), 
\end{align*}
and similar to \cref{eq-yinj}, we have 
\begin{align*}
	g^{\ind}(Y) = \sum_{B \subset [k]_2} g_B^{\ind}(Y_B). 
\end{align*}

Recall that $e(F)$ is the number of 2-stars in $F$ and $t(F)$ is the number of triangles in $F$. Let 
\begin{align*}
	\bar e(F) = \binom{v(F)}{2}^{-1} e(F), \quad \bar s(F) = \binom{v(F)}{3}^{-1} \frac{s(F)}{3}, \quad \bar t(F) = \binom{v(F)}{3}^{-1} t(F).
\end{align*}
Let
\begin{align*}
	N(F) = \frac{v(F)!}{ \lvert \Aut(F) \rvert	 } p^{e(F)} (1 - p)^{ \binom{v(F)}{2} - e(F) }.
\end{align*}
By \citet{janson1991asymptotic}, letting $B_1 = \{ (1,2) \}$, $B_2 = \{ (1,2), (1,3) \}$ and $B_3 = \{ (1,2), (1,3),(2,3) \}$, we have 
\begin{align*}
	g^{\ind}_{ B_1 } (y) & =   \frac{N(F)}{p(1 - p)} ( \bar e(F) - p ) (y - p), \\
	g^{\ind}_{ B_2 }(y_{12}, y_{13}) & =  \frac{N(F)}{p^2(1 - p)^2} (\bar s(F) - 2 p \bar e(F) + p^2) \\
												 & \quad  \times ( (y_{12} - p) (y_{13} - p) , \\
	g^{\ind}_{ B_3 }(y_{12}, y_{13}, y_{23}) & = \frac{N(F)}{p^3(1 - p)^3} (\bar t(F) - 3 p \bar s(F) + 3 p^2 \bar e(F) - p^3) \\
															   & \quad \times ( y_{12} - p ) (y_{13} - p) (y_{23} - p).
\end{align*}

We now consider the following three cases. 

\medskip 
{\bf\noindent Case 1.} If $e(F) \neq p \binom{v(F)}{2}$. 
In this case, we have 
$g^{\ind}_{B_1} \not\equiv 0$. Then, by \cref{thm-3.2}, we have \cref{eq-3.3-1} holds. 

\medskip 
{\bf\noindent Case 2.} If $\bar e(F) = p$ and $\bar s(F) \neq p^2$. 
In this case, we have 
\begin{align*}
	g^{\ind}_{B_1} \equiv 0, \quad g^{\ind}_{B_2} \not\equiv 0.
\end{align*}
However, the graph generated by $B_2$ is a 2-star, which is not strongly connected. Then, by \cref{thm-3.0}, we have \cref{eq-3.3-2} holds. 

\medskip 
{\bf\noindent Case 3.} If $\bar e(F) = p$, $\bar s(F) = p^2$ and $\bar t(F) \neq p^3$. 
In this case, we have 
\begin{align*}
	g^{\ind}_{B_1} \equiv 0, \quad g^{\ind}_{B_2} \equiv 0, \quad g^{\ind}_{B_3} \not\equiv 0	. 
\end{align*}
Because the graph generated by $B_3$ is a triangle, which is strongly connected. Then, by \cref{thm-3.2}, we have \cref{eq-3.3-1} holds.

\appendix 
\section{Proofs of some lemmas}%
\label{sec:proofs_of_some_lemmas}

\subsection{Proof of Lemma 5.1}%
\label{sub:proof_of_lemma_5_1}
\begin{proof}
[Proof of \cref{lem3.3}]
	We write $\{\alpha\} = \{ \alpha(1), \dots, \alpha(k) \} $ for any $\alpha = (\alpha(1), \dots, \alpha(k)) \in \mathcal{A}_{n,k}$.
 Also, write $r_{\alpha} = r( X_{\alpha(1)}, \dots, X_{\alpha(k)}; Y_{\alpha(1), \alpha(2)}, \dots, Y_{\alpha(k - 1), \alpha(k)} )$. 
 Now, observe that 
	\begin{equ}
	 \Var \biggl\{ \sum_{ \alpha \in \mathcal{I}_{n,k} } r_{\alpha} \biggr\}  = \frac{1}{\sigma_n^2}  \sum_{\alpha \in \mathcal{I}_{n,k}} \sum_{ \alpha' \in \mathcal{I}_{n,k} } \Cov \bigl( r_{\alpha}, r_{\alpha'}  \bigr) . 
        \label{l3.3-02}
	\end{equ}

	Note that if $\{\alpha\} \cap \{\alpha'\} = \emptyset$, then $r_{\alpha}$ and $r_{\alpha'}$ are independent, then clearly it follows that 
	\begin{equation}
		\Cov \bigl( r_{\alpha},
		r_{\alpha'} \bigr) = 0 
        \label{l3.3-03}
	\end{equation}
 if $\{\alpha\} \cap \{\alpha'\} = \emptyset$. 	
	If there exists $i \in \{1, \dots, n\}$ such that $\{\alpha\} \cap \{\alpha'\} = \{i\}$, then 
	\begin{equ}
		\Cov \bigl( r_{\alpha}, r_{\alpha'}) \bigr)
		 & = \E \Bigl\{ \Cov \bigl( r_{\alpha}, r_{\alpha'} \bigm\vert X_i \bigr) \Bigr\} + \Cov \bigl( \E \{ r_{\alpha} \vert X_i \} , \E \{ r_{\alpha'} \vert X_i\} \bigr).
        \label{l3.3-01}
	\end{equ}
	By independence, we have the first term of \cref{l3.3-01} is 0. For the second term, note that for any $i \in \{\alpha\}$, then $\E \{ r_{\alpha} \vert X_i \} = 0$, and thus the second term of \cref{l3.3-01} is also 0. Therefore, 
	\begin{equ}
		\Cov \bigl( r_{\alpha}, r_{\alpha'} \bigr) = 0, \text{ if $ \lvert \{\alpha\} \cap \{\alpha'\} \rvert = 1 $. }
        \label{l3.3-04}
	\end{equ}
	For any $\alpha$ and $\alpha'$ such that $\lvert \{\alpha\} \cap \{\alpha'\}\rvert \geq 2$, by the Cauchy inequality, we have 
	\begin{align*}
		\Cov \bigl( r_{\alpha}, r_{\alpha'} \bigr) \leq \Var \bigl(  r_{\alpha} \bigr) . 
	\end{align*}
	Recall that  $r_{\alpha}$ and $g(X_j)$  are orthogonal for every $j \in \{\alpha\}$.  
	By \cref{t3.1-01}, we have 
	\begin{align*}
		\Var \bigl(  r_{\alpha} \bigr) & =  \Var \bigl(  \func (X_{\alpha(1)}, \dots, X_{\alpha(k)}; Y_{\alpha(1),\alpha(2)}, \dots, Y_{\alpha(k-1),\alpha(k)} ) \bigr) \\
									   & \quad - \sum_{j \in \{\alpha\}} \E \{ f_1(X_j)^2 \} \\
									   & \leq \tau^{2}.
	\end{align*}
	Thus, it follows that 
	\begin{equ}
		\bigl\lvert \Cov \bigl(  r_{\alpha}, r_{\alpha'} \bigr)  \bigr\rvert& \leq \tau^{2}, \text{ if $ \lvert \{\alpha\} \cap \{\alpha'\}\rvert \geq 2 $. }
        \label{l3.3-05}
	\end{equ}
	Combining \cref{t31-sig,l3.3-02,l3.3-03,l3.3-04,l3.3-05}, we have 
	\begin{equ}
		\E \{U^2\} & \leq \frac{\tau^2}{\sigma_n^2} \sum_{\alpha \in \mathcal{I}_{n, k}} \sum_{\alpha' \in \mathcal{I}_{n,k}} \IN{ \lvert \{\alpha\} \cap \{\alpha'\} \rvert \geq 2 } \\
				   & \leq \frac{n\tau^2}{k^2\sigma_1^2}  \binom{n}{k}^{-1} \binom{k}{2} \binom{n-k}{k-2}\\ 
				   & \leq \frac{(k-1)^2 \tau^2}{2(n-1)\sigma_1^2}. 
        \label{l3.3-06}
	\end{equ}
	This proves \cref{l3.3-aa}. 

	Now we prove \cref{l3.3-bb}. Let $\mathcal{I}_{n,k}^{(i)} = \{ \alpha = \{ \alpha(1), \dots, \alpha(k) \} : \alpha(1) < \dots < \alpha(k), i \in \{\alpha\} \}$.  Note that 
	\begin{align*}
		U - U^{(i)}	 
		& =  \frac{1}{\sigma_n}\sum_{\alpha \in \mathcal{I}_{n,k}^{(i)} } r_{\{\alpha\}}^{(i)}. 
	\end{align*}
	where 
	\begin{align*}
		r_{\alpha}^{(i)} = r_{\alpha} -  r ( X_{\alpha(1)}^{(i)}, \dots, X_{\alpha(k)}^{(i)}; Y_{\alpha(1),\alpha(2)}, \dots, Y_{\alpha(k-1),\alpha(k)} ) .
	\end{align*}
	For each $\alpha$, by independence, we have 
	\begin{align*}
		\Var \bigl( r_{\alpha}^{(i)} \bigr)
		& = 2 \E \{  \Var \bigl( r_{\alpha} \bigm\vert X_j, j \in \{\alpha\} \setminus \{i\}, Y_{\alpha(1),\alpha(2)}, \dots, Y_{\alpha(k-1),\alpha(k)} \bigr) \} \\
		& \leq 2 \Var \bigl( r_{\alpha} \bigr) \leq 2 \tau^{2}.
	\end{align*}
	Similar to \cref{l3.3-06}, we have 
	\begin{align*}
		\E \{ ( U - U^{(i)} )^2 \} 
		& =  \frac{1}{\sigma_n^2}\sum_{\alpha \in \mathcal{I}_{n,k}^{(i)}} \sum_{\alpha' \in \mathcal{I}_{n,k}^{(i)}} \Cov \bigl( r_{\alpha}^{(i)}, r_{ \alpha' }^{(i)} \bigr) \\
		& \leq \frac{2 n \tau^2 }{k^2 \sigma_1^2} \binom{n}{k}^{-2} \sum_{\alpha \in \mathcal{I}_{n,k}^{(i)}} \sum_{\alpha' \in \mathcal{I}_{n,k}^{(i)}} \IN{ \lvert \{\alpha\} \cap \{\alpha'\} \rvert \geq 2 }  \\
		& \leq \frac{2n (k-1) \tau^{2}}{k^2 \sigma_1^2} \binom{n}{k}^{-2} \binom{n-1}{k-1}\binom{n-k}{k-2} \\
		& \leq \frac{2 (k - 1)^2 \tau^{2}}{n(n - 1) \sigma_1^2}.
	\end{align*}
	This completes the proof. 
\end{proof}

\subsection{Proof of Lemma 5.3}%
\label{sub:proof_of_lemma_4_3}

Recall that $\{\alpha\} = \{ \alpha(1), \dots, \alpha(\ell) \}$ for $\alpha \in \mathcal{A}_{n,\ell}$.
To prove \cref{lem5.3}, we need the following lemma. 
\begin{lemma}
	\label{lem-cov1}
	Let $(A_1, B_1), (A_2,B_2)\in \mathcal{G}_{f,d}$, $(i,j) , (i',j') \in \mathcal{A}_{n,2}$, $\alpha_1 , \alpha_2 \in \mathcal{A}_{n,d}^{(i,j)}$ and $\alpha_1' , \alpha_2' \in \mathcal{A}_{n,d}^{(i',j')}$. Let 
	\begin{align*}
		s = | \{\alpha_1\} \cap \{ \alpha_2 \} |, \quad t =  | \{\alpha_1'\} \cap \{ \alpha_2' \} | .
	\end{align*}
	If $|(\{ \alpha_1 \} \cup \{\alpha_2\}) \cap ( \{\alpha_1'\} \cap \{\alpha_2'\} )| \leq 2d - (s + t)$, then
	\begin{equ}
		\Cov \Bigl\{ \xi_{\alpha_1(A_1, B_1)}^{(i,j)} \xi_{\alpha_2(A_2, B_2)}^{(i,j)} , \xi_{\alpha_1'(A_1, B_1)}^{(i',j')} \xi_{\alpha_2'(A_2, B_2)}^{(i',j')}\Bigr\} = 0. 
        \label{lc1-a}
	\end{equ}
\end{lemma}
\begin{proof}
[Proof of \cref{lem-cov1}]
	Let 
	\begin{equ}
		V_0 & =  \{\alpha_1\} \cap \{\alpha_2\}, & V_1 & = \{\alpha_1\} \setminus V_0, & V_2 & = \{\alpha_2\} \setminus V_0, & s & = \lvert V_0 \rvert, \\
		V_0' & =  \{\alpha_1'\} \cap \{\alpha_2'\}, & V_1' & = \{\alpha_1'\} \setminus V_0', & V_2' & = \{\alpha_2'\} \setminus V_0', 
		 & t &= \lvert V_0' \rvert. 
        \label{l3.0-05}
	\end{equ}
	Then, we have $V_1 \cap V_2 = \emptyset$, $V_1' \cap V_2' = \emptyset$, 
	$2 \leq s,t \leq d$. Without loss of generality, assume that $s \leq t$. 
	
If $2d - (s + t) = 0$, which is equivalent to $s = d, t = d$, then $\{\alpha_1\} = \{\alpha_2\}$ and $\{\alpha_1'\} = \{\alpha_2'\}$. If $\{a_1\} \cap \{a_1'\} = \emptyset$, then $(\xi_{\alpha_1(A_1, B_1)}^{(i,j)}, \xi_{\alpha_2(A_2, B_2)}^{(i,j)})$ and
	$(\xi_{\alpha_1'(A_1, B_1)}^{(i',j')}, \xi_{\alpha_2'(A_2, B_2)}^{(i',j')})$ are independent, which implies that \cref{lc1-a} holds.

If $2d - (s + t) > 0$ and  $ \lvert ( \{\alpha_1\} \cup \{ \alpha_2 \} ) \cap ( \{\alpha_1'\} \cup \{\alpha_2'\} ) \rvert
< 2d - (s + t)$, then there exists $r \in [n]$ such that $r \in (V_1'\cup V_2') \setminus( \{\alpha_1,\alpha_2\} )$. Now, assume that $r \in V_2' \setminus ( \{\alpha_1,\alpha_2\} )$ without loss of generality. Let 
	\begin{equ}
		\mathcal{F}_r = \sigma (X_p, Y_{p,q} , p , q \in [n] \setminus \{r\}) \vee \sigma(Y_{i',j'}').
        \label{l3.0-0F}
	\end{equ}
	Therefore, we have $\xi_{\alpha_1(A_1, B_1)}^{(i,j)}, \xi_{\alpha_2(A_2, B_2)}^{(i,j)},\xi_{\alpha_1'(A_1, B_1)}^{(i',j')} \in \mathcal{F}_r$. 
	Then, by \cref{eq-pp1},  
	\begin{align}
		& \E \{  \xi_{\alpha_2'(A_2, B_2)}^{(i',j')} \mvert \mathcal{F}_r \} \nonumber \\
		& =  \E \Bigl\{ f_{G_1} \bigl(X_{\alpha_2'(A_1, B_1)}; Y_{\alpha_2'(A_1, B_1)}\bigr) - f_{G_1} \bigl(X_{\alpha_2'(A_1, B_1)}; Y_{\alpha_2'(A_1, B_1)}^{(i',j')}	\bigr) \Bigm\vert \mathcal{F}_r \Bigr\}
		\nonumber \\
		& = 0. \label{l3.0-0bb}
	\end{align}
	Hence, 
	\begin{align*}
		\E \Bigl\{ \xi_{\alpha_1'(A_1, B_1)}^{(i',j')} \xi_{\alpha_2'(A_2, B_2)}^{(i',j')} \mvert \mathcal{F}_r \Bigr\} = 0,
		\intertext{which further implies that}\E \Bigl\{\xi_{\alpha_1'(A_1, B_1)}^{(i',j')} \xi_{\alpha_2'(A_2, B_2)}^{(i',j')}\Bigr\} = 0, 
	\end{align*}
	and
	\begin{equ}
		\MoveEqLeft \Cov \bigl\{ \xi_{\alpha_1(A_1, B_1)}^{(i,j)} \xi_{\alpha_2(A_2, B_2)}^{(i,j)},\xi_{\alpha_1'(A_1, B_1)}^{(i',j')} \xi_{\alpha_2'(A_2, B_2)}^{(i',j')} \bigr\}\\
		& = \E \{ \xi_{\alpha_1(A_1, B_1)}^{(i,j)} \xi_{\alpha_2(A_2, B_2)}^{(i,j)}\xi_{\alpha_1'(A_1, B_1)}^{(i',j')} \xi_{\alpha_2'(A_2, B_2)}^{(i',j')} \}\\
		& = \E \Bigl\{ \E \bigl\{  \xi_{\alpha_1(A_1, B_1)}^{(i,j)} \xi_{\alpha_2(A_2, B_2)}^{(i,j)}\xi_{\alpha_1'(A_1, B_1)}^{(i',j')} \xi_{\alpha_2'(A_2, B_2)}^{(i',j')} \mvert \mathcal{F}_r \bigr\} \Bigr\} \\
		& = \E \Bigl\{  \xi_{\alpha_1(A_1, B_1)}^{(i,j)} \xi_{\alpha_2(A_2, B_2)}^{(i,j)}\xi_{\alpha_1'(A_1, B_1)}^{(i',j')}  \E \bigl\{ \xi_{\alpha_2'(A_2, B_2)}^{(i',j')} \mvert \mathcal{F}_r \bigr\} \Bigr\}\\
		& = 0.
        \label{l3.0-01}
	\end{equ}
	If $2d - (s + t) > 0$ and  $\lvert ( \{\alpha_1\} \cup \{ \alpha_2 \} ) \cap ( \{\alpha_1'\} \cap \{\alpha_2'\} ) \rvert
	= 2d - (s + t)$, then either the following two conditions holds: (a) there exists $r \in V_1' \cup V_2'\setminus( \{\alpha_1\} \cup \{\alpha_2\} )$ or (b) $V_0 \cap V_0' = \emptyset$. If (a) holds, then following a similar argument that leading to
	\cref{l3.0-01}, we have 
	\cref{lc1-a} holds. 

	If (b) is true, letting $\mathcal{F} = \sigma ( X , \{Y_{p,q}: p, q \in V_1 \cup V_2 \cup V_1' \cup V_2'\}) $, we have conditional on $\mathcal{F}$, $(\xi_{\alpha_1(A_1, B_1)}^{(i,j)}, \xi_{\alpha_2(A_2, B_2)}^{(i,j)})$ is conditionally independent of $(\xi_{\alpha_1'(A_1, B_1)}
	^{(i',j')}, \xi_{\alpha_2'(A_2, B_2)}^{(i',j')})$, and thus, 
	\begin{multline*}
		\Cov \Bigl\{ \xi_{\alpha_1(A_1, B_1)}^{(i,j)} \xi_{\alpha_2(A_2, B_2)}^{(i,j)}, \xi_{\alpha_1'(A_1, B_1)}^{(i',j')} \xi_{\alpha_2'(A_2, B_2)}^{(i',j')} \Bigr\}\\ 
		= \Cov \biggl\{ \E \Bigl\{ \xi_{\alpha_1(A_1, B_1)}^{(i,j)} \xi_{\alpha_2(A_2, B_2)}^{(i,j)} \mvert \mathcal{F} \Bigr\}, \E \Bigl\{ \xi_{\alpha_1'(A_1, B_1)}^{(i',j')}\xi_{\alpha_2'(A_2, B_2)}^{(i',j')} \mvert \mathcal{F} \Bigr\} \biggr\}. 
	\end{multline*}
	Without loss of generality, we assume that $ V_1 \cup V_2 \cup V_1' \cup V_2' \neq \emptyset$, otherwise the argument is even simpler. Moreover, we may assume that $V_1 \neq \emptyset$.
	Let $\mathcal{F}_0 = \sigma( Y_{i,j}', Y_{p,q} : p, q \in V_0 )$, and we have $\xi_{\alpha_1(A_1, B_1)}^{(i,j)}$ and $\xi_{\alpha_2(A_2, B_2)}^{(i,j)}$ are conditionally independent given $\mathcal{F} \vee \mathcal{F}_0$. Moreover, by \cref{eq-pp1}, 
	\begin{math}
		\E \{ \xi_{\alpha_1(A_1, B_1)}^{(i,j)} \mvert \mathcal{F}\vee \mathcal{F}_0\} = \E \{ \xi_{\alpha_2(A_2, B_2)}^{(i,j)} \mvert \mathcal{F}\vee \mathcal{F}_0\}
		 = 0, 
	\end{math}
	and thus $\E \{ \xi_{\alpha_1(A_1, B_1)}^{(i,j)} \xi_{\alpha_2(A_2, B_2)}^{(i,j)} \mvert \mathcal{F} \} = 0$. Therefore, we have  under the condition (b),
	\begin{align}
		\text{
			$\Cov \{\xi_{\alpha_1(A_1, B_1)}^{(i,j)} \xi_{\alpha_2(A_2, B_2)}^{(i,j)}, \xi_{\alpha_1'(A_1, B_1)}^{(i',j')}\xi_{\alpha_2'(A_2, B_2)}^{(i',j')}\} = 0$. 
		}
		\label{l3.0-0b}
	\end{align}
	Combining \cref{l3.0-01,l3.0-0b} we prove that \cref{lc1-a} holds for $\lvert \{\alpha_1,\alpha_2\} \cap \{\alpha_1',\alpha_2'\} \rvert = 2d - (s + t)$.  This completes the proof. 
\end{proof}

\begin{proof}
	[Proof of \cref{lem5.3}]
	In this proof, we denote by $C$ a constant depending on $k$ and $d$, which may take different values in different places. 
	Note that $2 \leq s , t \leq d$, and 
	\begin{align}
		\MoveEqLeft \Var \Biggl\{ \sum_{(i,j)\in \mathcal{A}_{n,2}} \biggl(\sum_{\alpha_1 \in \mathcal{A}_{n,d}^{(i,j)} } \xi_{\alpha_1(A_1, B_1)}^{(i,j)}  \biggr) \biggl(\sum_{\alpha_2 \in \mathcal{A}_{n,d}^{(i,j)} } \xi_{\alpha_2(A_2, B_2)}^{(i,j)}  \biggr) \Biggr\}\nonumber \\
		& = \sum_{(i,j) \in \mathcal{A}_{n,2} \atop (i',j') \in \mathcal{A}_{n,2}} \sum_{\alpha_1 \in \mathcal{A}_{n, d}^{(i,j)} \atop \alpha_2 \in \mathcal{A}_{n,d}^{(i,j)}} \sum_{\alpha_1' \in \mathcal{A}_{n,d}^{(i',j')} \atop \alpha_2' \in \mathcal{A}_{n,d}^{(i',j')}}   
		\Cov \bigl\{ \xi_{\alpha_1(A_1, B_1)}^{(i,j)} \xi_{\alpha_2(A_2, B_2)}^{(i,j)}, \xi_{\alpha_1'(A_1, B_1)}^{(i',j')} \xi_{\alpha_2'(A_2, B_2)}^{(i',j')}  \bigr\} \nonumber \\
		& = \sum_{s, t = 2}^{d} \sum_{(i,j) \in \mathcal{A}_{n,2} \atop (i',j') \in \mathcal{A}_{n,2}} \sum_{\alpha_1 \in \mathcal{A}_{n,d}^{(i,j)} \atop \alpha_2 \in \mathcal{A}_{n,d}^{(i,j)}} \sum_{\alpha_1' \in \mathcal{A}_{n,d}^{(i',j')} \atop \alpha_2' \in \mathcal{A}_{n,d}^{(i',j')}}  
	 \\ & \quad \times 
		\Cov \bigl\{ \xi_{\alpha_1(A_1, B_1)}^{(i,j)} \xi_{\alpha_2(A_2, B_2)}^{(i,j)}, \xi_{\alpha_1'(A_1, B_1)}^{(i',j')} \xi_{\alpha_2'(A_2, B_2)}^{(i',j')}  \bigr\}
		 \IN{O_{s,t}},
        \label{l3.0-00}
	\end{align}
where $O_{s,t} = \{ \lvert \{ \alpha_1 \} \cap \{\alpha_2\} \rvert = s\} \cap \{ \lvert \{\alpha_1'\} \cap \{\alpha_2'\} \rvert = t \}$. 
If $\lvert ( \{\alpha_1\} \cup \{ \alpha_2 \} ) \cap ( \{\alpha_1'\} \cap \{\alpha_2'\} ) \rvert \leq 2d - (s + t)$, by \cref{lc1-a} in \cref{lem-cov1}, we have 
	\begin{align*}
		\Cov \bigl\{ \xi_{\alpha_1(A_1, B_1)}^{(i,j)} \xi_{\alpha_2(A_2, B_2)}^{(i,j)}, \xi_{\alpha_1'(A_1, B_1)}^{(i',j')} \xi_{\alpha_2'(A_2, B_2)}^{(i',j')}  \bigr\} = 0. 
	\end{align*}
	If $\lvert ( \{\alpha_1\} \cup \{ \alpha_2 \} ) \cap ( \{\alpha_1'\} \cap \{\alpha_2'\} ) \rvert > 2d - (s + t)$, then, recalling that $(\xi_{\alpha_1(A_1, B_1)}^{(i,j)}, \xi_{\alpha_2(A_2, B_2)}^{(i,j)}) \stackrel{d.}{=} (\xi_{\alpha_1'(A_1, B_1)}^{(i',j')}, \xi_{\alpha_2'(A_2, B_2)}^{(i',j')})$, we have  
	\begin{align}
		\MoveEqLeft \bigl\lvert \Cov \{ \xi_{\alpha_1(A_1, B_1)}^{(i,j)} \xi_{\alpha_2(A_2, B_2)}^{(i,j)}, \xi_{\alpha_1'(A_1, B_1)}^{(i',j')} \xi_{\alpha_2'(A_2, B_2)}^{(i',j')} \} \bigr\rvert \nonumber \\ 
		& \leq \E \{ (\xi_{\alpha_1(A_1, B_1)}^{(i,j)})^2 (\xi_{\alpha_2(A_2, B_2)}^{(i,j)})^2\} \nonumber \\
		& \leq C \bigl(\E \{  \func_{A_1,B_1}^4 (X_{ \alpha_1(A_1, B_1)}; Y_{\alpha_1(A_1, B_1) })  \}  + \E \{  \func_{A_1,B_1}^4 (X_{ \alpha_2(A_2, B_2)}; Y_{\alpha_2(A_2, B_2) })  \}\bigr)\nonumber \\	
		& \leq C \tau^4. 		
		\label{l3.0-0C}
	\end{align}
	Therefore, with 
	\begin{align*}
		O_1 = \{ \lvert ( \{\alpha_1\} \cup \{\alpha_2\} ) \cap (\{ \alpha_1' \} \cup \{\alpha_2'\}) \rvert > 2d - (s + t)	\}, 
	\end{align*}
	we have 
	\begin{align*}
		\MoveEqLeft\Var \Biggl\{ \sum_{(i,j)\in \mathcal{A}_{n,2}} \biggl( \sum_{\alpha \in \mathcal{A}_{n,d}^{(i,j)}} \xi_{\alpha(A_1,B_1)}^{(i,j)} \biggr) \biggl( \sum_{\alpha \in \mathcal{A}_{n,d}^{(i,j)}} \xi_{\alpha(A_1,B_1)}^{(i,j)} \biggr)\Biggr\} \\
		& \leq C \tau^4 \sum_{s, t = 0}^{d} \sum_{(i,j) \in \mathcal{A}_{n, 2} \atop (i',j') \in \mathcal{A}_{n,2}} \sum_{\alpha_1 \in \mathcal{A}_{n, d}^{(i,j)} \atop \alpha_2 \in \mathcal{A}_{n,d}^{(i,j)}} \sum_{\alpha_1' \in \mathcal{A}_{n,d}^{(i',j')} \atop \alpha_2' \in \mathcal{A}_{n,d}^{(i',j')}} 
		 \IN
		{ O_1 \cap O_{s,t}}  \\
		& \leq C \tau^4\sum_{s,t = 0}^{d}  n^{ (2d - s) + (2d - t) - (2d - s - t + 1) }\\
		& \leq C n^{2d - 1} \tau^4. \qedhere
	\end{align*}
\end{proof}

\begin{proof}
	[Proof of \cref{lem5.5}]
	If $k < d + 1$, then it follows that $\xi_{\alpha(G)} = 0$ for all $G \in \Gamma_{d + 1}$ and $\alpha \in \mathcal{A}_{n,d + 1}$.
	Therefore, we assume $k \geq d + 1$ without loss of generality. 

	Observe that 
	\begin{equ}
		\MoveEqLeft 
		\Var \Biggl\{ \sum_{(i,j) \in \mathcal{A}_{n,2}} \biggl( \sum_{\alpha_1 \in \mathcal{A}_{n,v_1}^{(i,j)}} \xi_{\alpha_1(A_1, B_1)}^{(i,j)} \biggr) \biggl\lvert \sum_{\alpha_2 \in \mathcal{A}_{n,v_2}^{(i,j)}} \xi_{\alpha_2(A_2, B_2)}^{(i,j)} \biggr\rvert \Biggr\}
		\\
		& = \sum_{(i,j) \in \mathcal{A}_{n,2}} \sum_{(i',j') \in \mathcal{A}_{n,2}}\Cov \Biggl\{ \biggl( \sum_{\alpha_1 \in \mathcal{A}_{n,v_1}^{(i,j)}} \xi_{\alpha_1(A_1, B_1)}^{(i,j)} \biggr) \biggl\lvert \sum_{\alpha_2 \in \mathcal{A}_{n,v_2}^{(i,j)}} \xi_{\alpha_2(A_2, B_2)}^{(i,j)} \biggr\rvert, \\
		& \hspace{5cm} \biggl( \sum_{\alpha_1' \in \mathcal{A}_{n,v_1}^{(i',j')}} \xi_{\alpha_1'(A_1, B_1)}^{(i',j')} \biggr) \biggl\lvert \sum_{\alpha_2' \in \mathcal{A}_{n,v_2}^{(i',j')}} \xi_{\alpha_2'(A_2, B_2)}^{(i',j')} \biggr\rvert \Biggr\}.
        \label{eq-A21}
	\end{equ}
	Letting
	\begin{align*}
		\mathcal{F}_1 = \sigma(X) \vee \sigma ( Y_{p,q}, Y_{p,q}' : \{p,q\} \neq \{i,j\} ), 
	\end{align*}
	and noting that 
	\begin{align*}
		\biggl( \sum_{\alpha_1 \in \mathcal{A}_{n,v_1}^{(i,j)}} \xi_{\alpha_1(A_1, B_1)}^{(i,j)} \biggr) \biggl\lvert \sum_{\alpha_2 \in \mathcal{A}_{n,v_2}^{(i,j)}} \xi_{\alpha_2(A_2, B_2)}^{(i,j)} \biggr\rvert
	\end{align*}
	is anti-symmetric with respect to $(Y_{ij}, Y_{ij}')$, we have 
	\begin{align*}
		\E \Biggl\{ \biggl( \sum_{\alpha_1 \in \mathcal{A}_{n,v_1}^{(i,j)}} \xi_{\alpha_1(A_1, B_1)}^{(i,j)} \biggr) \biggl\lvert \sum_{\alpha_2 \in \mathcal{A}_{n,v_2}^{(i,j)}} \xi_{\alpha_2(A_2, B_2)}^{(i,j)} \biggr\rvert \Biggr\} = 0. 
	\end{align*}
	Now, we consider the following two cases. First, if $\{i,j\} \neq \{i',j'\}$, we have 
	\begin{align*}
		\biggl( \sum_{\alpha_1' \in \mathcal{A}_{n,v_1}^{(i',j')}} \xi_{\alpha_1'(A_1, B_1)}^{(i',j')} \biggr) \biggl\lvert \sum_{\alpha_2' \in \mathcal{A}_{n,v_2}^{(i',j')}} \xi_{\alpha_2'(A_2, B_2)}^{(i',j')} \biggr\rvert \quad \text{is $\mathcal{F}_1$ measurable}
	\end{align*}
	and by anti-symmetry again, 
	\begin{align*}
		\E \Biggl\{ \biggl( \sum_{\alpha_1 \in \mathcal{A}_{n,v_1}^{(i,j)}} \xi_{\alpha_1(A_1, B_1)}^{(i,j)} \biggr) \biggl\lvert \sum_{\alpha_2 \in \mathcal{A}_{n,v_2}^{(i,j)}} \xi_{\alpha_2(A_2, B_2)}^{(i,j)} \biggr\rvert \Biggm\vert \mathcal{F}_1 \Biggr\} = 0.
	\end{align*}
	Therefore, 
	\begin{multline}
			\Cov \Biggl\{ \biggl( \sum_{\alpha_1 \in \mathcal{A}_{n,v_1}^{(i,j)}} \xi_{\alpha_1(A_1, B_1)}^{(i,j)} \biggr) \biggl\lvert \sum_{\alpha_2 \in \mathcal{A}_{n,v_2}^{(i,j)}} \xi_{\alpha_2(A_2, B_2)}^{(i,j)} \biggr\rvert, 
				\\
				 \quad \biggl( \sum_{\alpha_1' \in \mathcal{A}_{n,v_1}^{(i',j')}}
		\xi_{\alpha_1'(A_1, B_1)}^{(i',j')} \biggr) \biggl\lvert \sum_{\alpha_2' \in \mathcal{A}_{n,v_2}^{(i',j')}} \xi_{\alpha_2'(A_2, B_2)}^{(i',j')} \biggr\rvert \Biggr\} = 0
        \label{eqA23}
	\end{multline}
	for $\{i,j\} \neq \{i',j'\}$.

	It suffices to consider the case where $\{i,j\} = \{i',j'\}$. Observe that 
	\begin{equ}
		\MoveEqLeft
		\Cov \Biggl\{ \biggl( \sum_{\alpha_1 \in \mathcal{A}_{n,v_1}^{(i,j)}} \xi_{\alpha_1(A_1, B_1)}^{(i,j)} \biggr) \biggl\lvert \sum_{\alpha_2 \in \mathcal{A}_{n,v_2}^{(i,j)}} \xi_{\alpha_2(A_2, B_2)}^{(i,j)} \biggr\rvert, \biggl( \sum_{\alpha_1' \in \mathcal{A}_{n,v_1}^{(i,j)}}
		\xi_{\alpha_1'(A_1, B_1)}^{(i,j)} \biggr) \biggl\lvert \sum_{\alpha_2' \in \mathcal{A}_{n,v_2}^{(i,j)}} \xi_{\alpha_2'(A_2, B_2)}^{(i,j)} \biggr\rvert \Biggr\}\\
		& = \E \Biggl\{ \biggl( \sum_{\alpha_1 \in \mathcal{A}_{n,v_1}^{(i,j)}} \xi_{\alpha_1(A_1, B_1)}^{(i,j)} \biggr) \biggl( \sum_{\alpha_1' \in \mathcal{A}_{n,v_1}^{(i,j)}}
		\xi_{\alpha_1'(A_1, B_1)}^{(i,j)} \biggr) \biggl\lvert \biggl(\sum_{\alpha_2 \in \mathcal{A}_{n,v_2}^{(i,j)}} \xi_{\alpha_2(A_2, B_2)}^{(i,j)}\biggr) \biggl(\sum_{\alpha_2' \in \mathcal{A}_{n,v_2}^{(i,j)}} \xi_{\alpha_2'(A_2, B_2)}^{(i,j)}\biggr) \biggr\rvert \Biggr\}
		\\
		& = \sum_{\alpha_1 \in \mathcal{A}_{n,v_1}^{(i,j)}} \sum_{\alpha_1' \in \mathcal{A}_{n,v_1}^{(i,j)}}\E \Biggl\{   \xi_{\alpha_1(A_1, B_1)}^{(i,j)}   
		\xi_{\alpha_1'(A_1, B_1)}^{(i,j)}  \biggl\lvert \sum_{\alpha_2 \in \mathcal{A}_{n,v_2}^{(i,j)}} \sum_{\alpha_2' \in \mathcal{A}_{n,v_2}^{(i,j)}} \xi_{\alpha_2(A_2, B_2)}^{(i,j)}  \xi_{\alpha_2'(A_2, B_2)}^{(i,j)} \biggr\rvert \Biggr\}
		.
        \label{eqA22}
	\end{equ}
	Let $H_1 = \{\alpha_1\} \setminus \{\alpha_1'\}$ and $H_1' = \{\alpha_1'\}\setminus \{\alpha_1\}$. 
	Let $t = \lvert \alpha_1 \cap \alpha_1' \rvert$, and then we have $2 \leq t \leq v_1$. Now, as 
	\begin{align*}
		\sum_{\alpha_2 \in \mathcal{A}_{n,v_2}^{(i,j)}}\sum_{\alpha_2' \in \mathcal{A}_{n,v_2}^{(i,j)}} \xi_{\alpha_2(A_2, B_2)}^{(i,j)}  \xi_{\alpha_2'(A_2, B_2)}^{(i,j)} & =  \sum_{\alpha_2, \alpha_2' \in \mathcal{A}_1} \xi_{\alpha_2(A_2, B_2)}^{(i,j)} \xi_{\alpha_2'(A_2, B_2)}^{(i,j)}
		\\
		& \quad +
		\sum_{\alpha_2, \alpha_2' \in \mathcal{A}_2} \xi_{\alpha_2(A_2, B_2)}^{(i,j)} \xi_{\alpha_2'(A_2, B_2)}^{(i,j)}, 
	\end{align*}
	where $\mathcal{A}_1 = \{\alpha_2,\alpha_2' \in \mathcal{A}_{n,v_2}^{(i,j)}: (H_1 \cup H_1') \setminus \{ \alpha_2, \alpha_2' \} \neq \emptyset  \}$ and $\mathcal{A}_2 = \{\alpha_2,\alpha_2' \in
	\mathcal{A}_{n,v_2}^{(i,j)}: (H_1 \cup H_1') \setminus \{ \alpha_2, \alpha_2' \} = \emptyset\}$.
	If there exists $r \in (H_1 \cup H_1') \setminus \{\alpha_2, \alpha_2'\}$, letting $\mathcal{F}_r = \sigma(X_p, Y_{p,q}, Y_{p,q}' : p,q \in [n] \setminus \{r\})$, then we have 
	\begin{align*}
		\sum_{\alpha_2, \alpha_2' \in \mathcal{A}_1} \xi_{\alpha_2(A_2, B_2)}^{(i,j)} \xi_{\alpha_2'(A_2, B_2)}^{(i,j)} \in \mathcal{F}_r, 
	\end{align*}
	and by orthogonality, we have 
	\begin{align*}
		\IE \bigl\{ \xi_{\alpha_1(A_1, B_1)}^{(i,j)}   
		\xi_{\alpha_1'(A_1, B_1)}^{(i,j)} \vert \mathcal{F}_r \bigr\} = 0. 
	\end{align*}
	Therefore, we have 
	\begin{align*}
		\E \Biggl\{   \xi_{\alpha_1(A_1, B_1)}^{(i,j)}   
		\xi_{\alpha_1'(A_1, B_1)}^{(i,j)}  \biggl\lvert \sum_{\alpha_2, \alpha_2' \in \mathcal{A}_1}  \xi_{\alpha_2(A_2, B_2)}^{(i,j)}  \xi_{\alpha_2'(A_2, B_2)}^{(i,j)} \biggr\rvert \Biggr\} = 0. 
	\end{align*}
	Hence, by Cauchy's inequality, we have 
	\begin{align*}
		\MoveEqLeft 
		\Biggl\lvert \E \Biggl\{   \xi_{\alpha_1(A_1, B_1)}^{(i,j)}   
		\xi_{\alpha_1'(A_1, B_1)}^{(i,j)}  \biggl\lvert \sum_{\alpha_2, \alpha_2' \in \mathcal{A}_{n,v_2}^{(i,j)}}  \xi_{\alpha_2(A_2, B_2)}^{(i,j)}  \xi_{\alpha_2'(A_2, B_2)}^{(i,j)} \biggr\rvert \Biggr\} \Biggr\rvert\\
		& \leq \E \Biggl\{   \bigl\lvert \xi_{\alpha_1(A_1, B_1)}^{(i,j)}   
		\xi_{\alpha_1'(A_1, B_1)}^{(i,j)} \bigr\rvert  \biggl\lvert \sum_{\alpha_2, \alpha_2' \in \mathcal{A}_{2}}  \xi_{\alpha_2(A_2, B_2)}^{(i,j)}  \xi_{\alpha_2'(A_2, B_2)}^{(i,j)} \biggr\rvert \Biggr\}\\
		& \leq C \tau^2 \sqrt{\E \Biggl\{   \biggl\lvert \sum_{\alpha_2, \alpha_2' \in \mathcal{A}_{2}}  \xi_{\alpha_2(A_2, B_2)}^{(i,j)}  \xi_{\alpha_2'(A_2, B_2)}^{(i,j)} \biggr\rvert^2 \Biggr\}}.
	\end{align*}
	Following the similar argument in the proof of \cref{lem5.3}, and recalling that $\{\alpha_1 \cap \alpha_1'\} = t$ and $\lvert \mathcal{A}_2 \rvert \leq Cn^{2( t-2)}( n^{v_2 - v_1} \vee 1 )$, we have  
	\begin{align*}
		\E \Biggl\{   \biggl\lvert \sum_{\alpha_2, \alpha_2' \in \mathcal{A}_{2}}  \xi_{\alpha_2(A_2, B_2)}^{(i,j)}  \xi_{\alpha_2'(A_2, B_2)}^{(i,j)} \biggr\rvert^2 \Biggr\} & \leq C n^{2( t-2)}( n^{v_2 - v_1} \vee 1 )\tau^4. 
	\end{align*}
	Therefore, we have 
	\begin{align*}
		\Biggl\lvert \E \Biggl\{   \xi_{\alpha_1(A_1, B_1)}^{(i,j)}   
		\xi_{\alpha_1'(A_1, B_1)}^{(i,j)}  \biggl\lvert \sum_{\alpha_2, \alpha_2' \in \mathcal{A}_{n,v_2}^{(i,j)}}  \xi_{\alpha_2(A_2, B_2)}^{(i,j)}  \xi_{\alpha_2'(A_2, B_2)}^{(i,j)} \biggr\rvert \Biggr\} \Biggr\rvert\leq C n^{2( t-2)}( n^{v_2 - v_1} \vee 1 )\tau^4. 
	\end{align*}
	Substituting the foregoing inequality to \cref{eqA22}, we have 
	\begin{multline}
		\sum_{(i,j) \in \mathcal{A}_{n,2}}\Cov \Biggl\{ \biggl( \sum_{\alpha_1 \in \mathcal{A}_{n,v_1}^{(i,j)}} \xi_{\alpha_1(A_1, B_1)}^{(i,j)} \biggr) \biggl\lvert \sum_{\alpha_2 \in \mathcal{A}_{n,v_2}^{(i,j)}} \xi_{\alpha_2(A_2, B_2)}^{(i,j)} \biggr\rvert, \\
			\biggl( \sum_{\alpha_1' \in \mathcal{A}_{n,v_1}^{(i,j)}}
		\xi_{\alpha_1'(A_1, B_1)}^{(i,j)} \biggr) \biggl\lvert \sum_{\alpha_2' \in \mathcal{A}_{n,v_2}^{(i,j)}} \xi_{\alpha_2'(A_2, B_2)}^{(i,j)} \biggr\rvert \Biggr\}\\
		\leq C n^{2\max \{v_1,v_2\} - 2} \tau^4. 
		\label{eqA24}
	\end{multline}
	By \cref{eq-A21,eqA23,eqA24}, we complete the proof. 

\end{proof}

\subsection{Proof of \cref{lem5.4}}%
\label{sub:a3}

\cref{lem5.4} follows from a similar argument as that in the proof of \cref{lem5.3} and the following lemma. Let $\tilde{\mathcal{G}}_{f,\ell} = \{ (A, B) \in \mathcal{G}_{f,\ell} : G_{A,B} \text{ is strongly connected.} \}$ Now, as the function $g$ does not depend on $X$, we set
$A_m = \emptyset$ in the following lemma. With a slight abuse of notation, For $j = 1, 2$ and for $B_m \subset [k]_2$, let $G_m$ be the graph generated by $B_m$ and let $v_m$  be the number of nodes of $G_{m}$, and we write $B_m \in \mathcal{G}$ if $G_m \in \mathcal{G}$. 
\begin{lemma}
	\label{lem-cova3}
	Let $B_m\in \tilde{\mathcal{G}}_{f,d}
\cup \mathcal{G}_{f,d+ 1}$ for $m = 1, 2$. 
	Let $(i,j), (i',j') \in \mathcal{A}_{n,2}$, and let $\alpha_m \in \mathcal{A}_{n,v_m}^{(i,j)}$, $\alpha_m' \in \mathcal{A}_{n,v_m}^{(i',j')}$ for $m = 1, 2$. Let $s = \lvert \{\alpha_1\} \cap \{\alpha_2\} \rvert$ and $t = \lvert \{\alpha_1'\} \cap \{\alpha_2'\} \rvert$.
	For $m = 1, 2$, let
	$\gamma_m$ indicate that $B_m \in \tilde{\mathcal{G}}_{f,d} \cup \tilde{\mathcal{G}}_{f,d + 1}$. Then
	\begin{equ}
		\Cov \Bigl\{ \eta_{\alpha_1(B_1)}^{(i,j)} \eta_{\alpha_2(B_2)}^{(i,j)},\eta_{\alpha_1'( B_1)}^{(i',j')} \eta_{\alpha_2'(B_2)}^{(i',j')} \Bigr\} = 0
        \label{eq-a31}
	\end{equ}
	for $\lvert \{\alpha_1, \alpha_2\} \cap \{\alpha_1', \alpha_2'\} \rvert < v_1 + v_2 + \gamma_1 + \gamma_2 - (s + t)$.
\end{lemma}
\begin{proof}
	The proof is similar to that of \cref{lem-cov1}.
	
	Let $V_0, V_0', V_1, V_1',V_2,V_2'$ be defined as in \cref{l3.0-05}. Note that if $G_B$ has isolated nodes, then $\eta_{\alpha(B)} = 0$ for all $\alpha \in \mathcal{A}_{n, v_{B}}$, where $v_{B}$ is the number of nodes of the graph generated by the index set $B$. If $v_1 + v_2 = s + t$, then it follows that $\{\alpha_1\} = \{\alpha_2\}$ and $\{\alpha_1'\}
	= \{\alpha_2'\}$. If $\lvert \{\alpha_1\} \cap \{\alpha_1'\} \rvert < 2$, then $\eta_{\alpha_1(B_1)}^{(i,j)} \eta_{\alpha_2(B_2)}^{(i,j)}$ and $\eta_{\alpha_1'(B_1)}^{(i',j')} \eta_{\alpha_2'(B_2)}^{(i',j')}$ are independent, which further implies that 
	\cref{eq-a31} holds. 

	Now we consider the case where $v_1 + v_2 > s + t$. 
	If $\lvert \{ \alpha_1 , \alpha_2 \} \cap \{\alpha_1', \alpha_2'\} \rvert < v_1 + v_2 - (s + t)$, then following the same argument as that leading to \cref{l3.0-01}, we have \cref{eq-a31} holds. 
	
	If $G_1$ is connected and $\lvert \{ \alpha_1 , \alpha_2 \} \cap \{\alpha_1', \alpha_2'\} \rvert = v_1 + v_2 - (s + t)$, then either the following two conditions holds: (a) there exists $r \in V_2'\setminus( \{\alpha_1\} \cup \{\alpha_2\} \cup V_0' \cup V_1')$ or (b) $V_0 \cap V_0' = \emptyset$. If (a)
	holds, then following a similar argument as before, we have 
	\cref{eq-a31} holds. Now we consider that the case where (b) holds. Let
	$H_1 = \{ (p,q) : p \in V_0, q \in V_1 \}$ and 
	\begin{align*}
		\mathcal{F}_1 = \sigma ( Y_{p,q}, Y_{p,q}',  : \mathcal{A}_{n,2} \setminus H_1	 ). 
	\end{align*}
	By orthogonality, we have $\E  \{ \eta_{\alpha_1(B_1)}^{(i,j)}  \vert \mathcal{F}_1  \}	= 0$.

	Note that $\eta_{\alpha_2(B_2)}, \eta_{\alpha_1'(B_1)}, \eta_{\alpha_2'(B_2)} \in \mathcal{F}_1$, we have 
	\begin{align*}
		\MoveEqLeft \E \Bigl\{ \eta_{\alpha_1(B_1)}^{(i,j)} \eta_{\alpha_2(B_2)}^{(i,j)} \Bigr\}\\
		& = \E \Bigl\{ \eta_{\alpha_2(B_2)}^{(i,j)} \E  \Bigl\{ \eta_{\alpha_1(B_1)}^{(i,j)}  \Bigm\vert \mathcal{F}_1 \Bigr\}\Bigr\} = 0, \\
		\MoveEqLeft
		\Cov \Bigl\{ \eta_{\alpha_1(B_1)}^{(i,j)} \eta_{\alpha_2(B_2)}^{(i,j)},\eta_{\alpha_1'(B_1)}^{(i',j')} \eta_{\alpha_2'(B_2)}^{(i',j')} \Bigr\}\\
		&  = \E \Bigl\{\E  \Bigl\{ \eta_{\alpha_1(B_1)}^{(i,j)} \eta_{\alpha_2(B_2)}^{(i,j)}\eta_{\alpha_1'(B_1)}^{(i',j')} \eta_{\alpha_2'(B_2)}^{(i',j')} \Bigm\vert \mathcal{F}_1 \Bigr\}\\
		& = \E \Bigl\{ \eta_{\alpha_2(B_2)}^{(i,j)}\eta_{\alpha_1'(B_1)}^{(i',j')} \eta_{\alpha_2'(B_2)}^{(i',j')} \E  \Bigl\{ \eta_{\alpha_1(B_1)}^{(i,j)}  \Bigm\vert \mathcal{F}_1  \Bigr\}\Bigr\}	= 0 . 
	\end{align*}
	This proves \cref{eq-a31} for the case where $\lvert \{ \alpha_1 , \alpha_2 \} \cap \{\alpha_1', \alpha_2'\} \rvert = v_1 + v_2 - (s + t)$.

	Now, we further assume that $\gamma_1 = \gamma_2 = 1$. If $G_1$ or $G_2$ is a graph containing one single edge, then the proof is even simpler. Without loss of generality, we now assume that $G_m^{(r)}$ is connected for every $r \in [n]$ for $m = 1, 2$. We then
	prove that \cref{eq-a31} holds when $\lvert \{ \alpha_1 , \alpha_2 \} \cap \{\alpha_1', \alpha_2'\} \rvert = v_1 + v_2 - (s + t) + 1$. Under this condition, additional to (a) and (b), there is still another
	event that may happen: (c) there exists $r \in [n]$ such that $\{r\} = V_0 \cap V_0'$. As the cases (a) and (b) have been discussed, we only need to prove that \cref{eq-a31} holds under (c).

	As $\{i,j\} \subset V_0$, we have $s \geq 2$, and $V_0 \setminus \{r\}$ is not empty. 
	Let 
	\begin{align*}
		\mathcal{F}_2 = \sigma\{ Y_{p,q}, Y_{p,q}' : p \in V_1 \cup V_2 \cup V_1' \cup V_2', q \in V_1 \cup V_2 \cup V_1' \cup V_2' \cup \{r\} \}.
	\end{align*}
	Then, conditional on $\mathcal{F}_2$, we have $\eta_{\alpha_1(B_1)}^{(i,j)} \eta_{\alpha_2(B_2)}^{(i,j)}$ and $\eta_{\alpha_1'(B_1)}^{(i',j')} \eta_{\alpha_2'(B_2)}^{(i',j')}$ are conditionally independent. Hence, 
	\begin{align*}
		\MoveEqLeft
		\Cov \Bigl\{ \eta_{\alpha_1(B_1)}^{(i,j)} \eta_{\alpha_2(B_2)}^{(i,j)},\eta_{\alpha_1'(B_1)}^{(i',j')} \eta_{\alpha_2'(B_2)}^{(i',j')} \Bigr\}\\
		& = \Cov \Bigl\{ \E \{\eta_{\alpha_1(B_1)}^{(i,j)} \eta_{\alpha_2(B_2)}^{(i,j)} \vert \mathcal{F}_2 \}, \E \{ \eta_{\alpha_1'(B_1)}^{(i',j')} \eta_{\alpha_2'(B_2)}^{(i',j')} \vert \mathcal{F}_2\} \Bigr\}. 
	\end{align*}
	Letting 
	\begin{align*}
		\mathcal{F}_3 = \sigma\{ Y_{p,q}, Y_{p,q} : p \in V_0 \setminus \{r\}, q \in V_2 \cup \{r\} \}. 
	\end{align*}
	Now, if $G_1^{(r)}$ is connected for every $r \in [n]$, there is at least one edge in $G_1$ connecting $V_0 \setminus \{r\}$ and $V_1$, and thus 
	\begin{align*}
		\E \{\eta_{\alpha_1(B_1)}^{(i,j)} \eta_{\alpha_2(B_2)}^{(i,j)} \vert \mathcal{F}_2 \vee \mathcal{F}_3  \} 
		& = \eta_{\alpha_2(B_2)}^{(i,j)}\E \{\eta_{\alpha_1(B_1)}^{(i,j)}  \vert \mathcal{F}_2 \vee \mathcal{F}_3 \} = 0, 
	\end{align*}
	where the last equality follows from orthogonality. 	 
	Noting that $\mathcal{F}_2 \subset \mathcal{F}_3$, then $\E \{\eta_{\alpha_1(B_1)}^{(i,j)} \eta_{\alpha_2(B_2)}^{(i,j)} \vert \mathcal{F}_2 \} = 0$ and thus 
	\cref{eq-a31} holds.

\end{proof}

\section*{Acknowledgements}
The research is supported by Singapore Ministry of Education Academic Research Fund MOE 2018-T2-076.

\setlength{\bibsep}{0.5ex}
\def\bibfont{\small}

% \bibliographystyle{myplainnat}
% \bibliography{exchbib.bib}

\end{document}